\newfont{\Bbb}{msbm10 scaled\magstephalf}
\documentclass[reqno]{amsart}
\usepackage{stmaryrd}
\usepackage{amssymb}
\usepackage{amsfonts}
\usepackage{amsmath, amsthm, amscd, amsfonts}

\newtheorem{theorem}{Theorem}[section]
\newtheorem{lemma}[theorem]{Lemma}
\newtheorem{proposition}[theorem]{Proposition}
\newtheorem{corollary}[theorem]{Corollary}
\theoremstyle{definition}
\newtheorem{definition}[theorem]{Definition}
\newtheorem{example}[theorem]{Example}

\theoremstyle{remark}
\newtheorem{remark}[theorem]{Remark}
\numberwithin{equation}{section}
\begin{document}
\setcounter{page}{1}

\title[Ranks of Commutators and Generalized Semicommutators]
{Ranks of Commutators and Generalized Semicommutators of Quasihomogeneous Toeplitz Operators}

\author[X.T Dong and Z.H. Zhou]{Xing-Tang Dong and Ze-Hua Zhou}

\address{\newline Xing-Tang Dong (Corresponding author)\newline Department of Mathematics, Tianjin University, Tianjin 300354, P.R. China.}
\email{dongxingtang@163.com}

\address{\newline  Ze-Hua Zhou\newline Department of Mathematics, Tianjin University, Tianjin 300354, P.R. China.} \email{zehuazhoumath@aliyun.com;
zhzhou@tju.edu.cn}



\begin{abstract}We study the ranks of commutators and generalized semicommutators of Toeplitz operators with quasihomogeneous symbols
on both the harmonic Bergman space and the Bergman Space.
In particular, when one of quasihomogeneous symbols is the the form of $e^{ik\theta }r^{m}$, we first obtain specific sufficient and necessary conditions for commutators and generalized semicommutators to be finite rank. Then we make further efforts to determine the range of each finite rank commutator and generalized semicommutators, and consequently the explicit canonical form and the rank are obtained. Thus, the finite rank problem of commutators and generalized semicommutators of such special Toeplitz operators is completely solved. As applications, several interesting corollaries and nontrivial examples are given. Also, we show close connections of the finite rank problem between the harmonic Bergman space and Bergman space cases.
\end{abstract}
 \maketitle

\section{Introduction}

Let $dA$ denote the Lebesgue area measure on the unit disk $D$,
normalized so that the measure of $D$ equals $1$. The harmonic Bergman
space $L_h^2$ is the closed subspace of $L^{2}(D,dA)$ consisting of
the harmonic functions on $D$.
Given $z\in D$, let $K_{z}(w)=\frac{1}{(1-w\overline{z})^2}$ be the
well-known reproducing kernel for the Bergman space
$L^{2}_{a}$ consisting of all $L^2$-analytic functions on $D$. Since $L_h^2=L^{2}_{a}+\overline{L^{2}_{a}}$, it is
easily checked that the reproducing kernel $R_z$ in $L_h^2$ is $$R_z=K_{z}+\overline{K_{z}}-1.$$ Thus, each $R_z$ is real-valued and the
orthogonal projection $Q$ from $L^{2}(D,dA)$ onto $L_h^2$ can be represented by $$Qf(z)=\int_D \left(\frac {1}{(1-z\bar{w})^2} + \frac {1}{(1-\bar{z}w)^2}-1\right) f(w) dA(w)$$
for $f\in L^{2}(D,dA)$. Recall the well-known Bergman projection $P$ has the following integral formula
$$P f(z)= \int_{D}f(w)\overline{K_{z}(w)}dA(w),$$ and hence $Q$ can be rewritten as
$Qf=Pf+\overline{P\overline{f}}-Pf(0).$

For $u\in L^{1}(D,dA)$, we define an operator ${T}_{u}$ with symbol
$u$ on $L_h^2$ by
\begin{equation}\label{DeT}{T}_{u}h=Q(u h)\end{equation} for $h\in L_h^2$. This operator is
always densely defined on the polynomials and not bounded in
general. We are interested in the case where this densely defined
operator is bounded in the $L^{2}_{h}$ norm. In this paper we will consider the case $u$ is a T-function, following \cite{LSZ}.

\begin{definition} Let $F\in L^{1}(D,dA)$.

\begin{enumerate}
\item[(a)]
We say that $F$ is a T-function if the equation \eqref{DeT},
with $u=F$, defines a bounded operator on $L^{2}_{h}$.
\item[(b)]
If $F$ is a T-function, we write $T_{F}$ for the continuous
extension of the operator defined by \eqref{DeT}. We say
that $T_{F}$ is a Toeplitz operator if and only if $T_{F}$ is
defined in this way.
\item[(c)]
If there is an $r\in(0,1)$ such that $F$ is (essentially)
bounded on the annulus $\left\{z:r<|z|<1\right\}$, then we say that
$F$ is "nearly bounded".
\end{enumerate}
\end{definition}

Generally, the T-functions form a proper subset of $L^{1}(D,dA)$
which contains all bounded and "nearly bounded" functions.

A function $f$ is said to be quasihomogeneous of degree
$k\in\mathbb{Z}$ if
\[
f(re^{i\theta})=e^{ik\theta}\varphi (r),
\]
where $\varphi $ is a radial function. In this case the associated
Toeplitz operator $T_f$ is also called quasihomogeneous Toeplitz
operator of degree $k$. By a straightforward
deduction, one can see that $e^{ik\theta}\varphi(r)$ is a
T-function if and only if $\varphi(r)$ is a
T-function.

For two Toeplitz operators $T_{f_{1}}$ and $T_{f_{2}}$, we define the commutator and the semicommutator respectively by $$\left[T_{f_{1}}, T_{f_{2}}\right]=T_{f_{1}}T_{f_{2}}-T_{f_{2}}T_{f_{1}}$$ and $$\left(T_{f_{1}}, T_{f_{2}}\right]=T_{f_{1}}T_{f_{2}}-T_{f_1f_2}.$$
For another Toeplitz operator $T_f$, the generalized semicommutator is defined by $T_{f_{1}}T_{f_{2}}-T_{f}.$

A major goal in the theory of Toeplitz operators on the function spaces is to describe
the zero commutator or generalized semicommutator of given Toeplitz operators. The corresponding problems on the classical Hardy space were completely solved by Brown and Halmos in \cite{BH}. However, in the setting of the Bergman space (or the harmonic Bergman space), only some special Toeplitz operators were considered. For example, harmonic Toeplitz operators has been investigated in \cite{AhC, AxC, ChoeKL, CL, CL1, Z1} at an earlier time,
and recently there has been an increasing interest in
quasihomogeneous Toeplitz operators, see, e.g. \cite{BI, BL, BV, CR, DZ1, DZ2, DZ3, DZ4, DLZ5, LSZ, LZ1, LZ, QS1, QS2, V, ZD}. The natural question to ask is why the quasihomogeneous Toeplitz operator is becoming the focus in this field. Perhaps there are mainly three reasons for it:
\begin{enumerate}
  \item It was shown in \cite{CR} that any function
$f\in L^2({D},dA)$ has
the following polar decomposition:
$f(re^{i\theta })=\sum\limits_{k \in\mathbb{Z}}
{e^{ik\theta } f_k (r)}.$ Obviously, each element of the polar decomposition is quasihomogeneous. So in order to better understand and research general Toeplitz operators, a more complete study of quasihomogeneous Toeplitz operators is needed.
  \item Quasihomogeneous function can be regarded as the generalization of monomial $z^m\overline{z}^n$. As is generally known, monomial Toeplitz operator, particularly when the symbol is analytic (or co-analytic) monomial, plays an important role in exploring basic algebraic properties of Toeplitz operators. Therefore, the research of quasihomogeneous Toeplitz operators in itself is quite meaningful.
  \item There are many nontrivial and complex results about quasihomogeneous Toeplitz operators. For example, if thinking of analytic
functions being placed on the real axis, conjugate analytic on the imaginary
axis and the radial functions on the diagonal $y=x$ in the first quadrant, then it was showed in \cite{CR} that there will be many lines
parallel to the diagonal, "holding" a symbol that gives a Toeplitz operator
commuting with $T_{z^m\overline{z}^n}$ on $L^2_a$. Moreover, each of these lines will hold no more than one such
symbol.
\end{enumerate}

In this paper, we continue to further study the property of the commutator or generalized semicommutator of given Toeplitz operators, i.e., finite rank problem.
Recall that an operator $A$ on a Hilbert space $H$ is said to have finite rank if the closure of $Ran(A)$ which is the range
of the operator has finite dimension. For a bounded finite rank operator $A$ on $H$, define $\emph{rank}(A)=\emph{dim}\ Ran(A)$.

On the Hardy space of the unit disk, the problem of determining when
the commutator or semicommutator has finite rank has been completely solved in \cite{ACD, DingZ}.
Also, the same problem has been studied on $L^{2}_{a}$. Guo, Sun and Zheng \cite{GSZ} showed that
there is no nonzero finite rank commutator or semicommutator of harmonic Toeplitz operators.
Then $\check{\textrm{C}}\textrm{u}\check{\textrm{c}}\textrm{kovi}\acute{\textrm{c}}$ \cite{Cu} studied the finite rank perturbation of the Brown-Halmos type results involving products of harmonic Toeplitz operators, which is called finite rank problem of generalized semicommutator in this paper.
Unlike the harmonic symbol case, $\check{\textrm{C}}\textrm{u}\check{\textrm{c}}\textrm{kovi}\acute{\textrm{c}}$ and Louhichi \cite{CuL} later
showed that commutators and semicommutators of two quasihomogeneous Toeplitz operators of opposite degrees can be nonzero finite rank operators.
More recently, $\check{\textrm{C}}\textrm{u}\check{\textrm{c}}\textrm{kovi}\acute{\textrm{c}}$
and Le \cite{CuLe} obtained results on finite rank commutators and semicommutators of harmonic Toeplitz operators on Bergman space of polyanalytic functions on the unit disk. On the contrast to the Hardy space and Bergman space cases, the conditions on harmonic Bergman space are quite different. Chen, Koo and Lee \cite{CKL} showed that the commutator of any two Toeplitz operators with general symbols can't
have an odd rank.

Obviously, only a few results of finite rank problem were known on both $L_a^2$ and $L_h^2$, and many seemingly simple and yet very complicated problems are still open. Take the commutator of monomial Toeplitz operators for example:
When the commutator $\left[T_{e^{ik_1\theta }r^{m_1}},T_{e^{ik_2\theta }r^{m_2}}\right]$ on $L_a^2$ (or on $L_h^2$) has finite rank remain to be solved. As one of interesting applications of our main theorem, we will give a complete characterization for it, including the range, the rank and the explicit canonical form of each finite rank commutator.

The paper is organized as follows. Motivated by \cite{CKL},
we first give some quite unexpected results in Section 2: The generalized semicommutator $T_{f_{1}}T_{f_{2}}-T_{f}$ with general symbols on $L_h^2$ has finite rank if and only if the anti-ordered generalized semicommutator $T_{f_{2}}T_{f_{1}}-T_{f}$ has finite rank. As a consequence, the finite rank generalized semicommutator $T_{f_{1}}T_{f_{2}}-T_{f}$ on $L_h^2$ implies that the commutator $\left[T_{f_{1}},T_{f_{2}}\right]$ must be finite rank. Obviously, this high correlation will be especially useful when we study the rank of the generalized semicommutator.

Returning to the quasihomogeneous symbol case, Section 3 contains some lemmas which will be used later.
In Sections 4, we discuss finite rank commutators of quasihomogeneous Toeplitz
operators on $L_h^2$. When one of symbols is the the form of $e^{ik\theta }r^{m}$ with $(k,m)\in \mathbb{Z}\times\left[-1, +\infty\right)$, the results have some aspects we wish to emphasize. Our results involve not only several specific necessary and sufficient conditions for commutators to be finite rank, but also the range, the rank and the explicit canonical form of the finite rank commutator in each condition. Much effort has gone into determining the range of each finite rank commutator. In fact, we first get the most likely range of finite rank commutator $\left[T_{e^{ik_1\theta} r^m},T_{e^{ik_2\theta }\varphi}\right]$. Then using the symmetry number of the commutator together masterly handling of the coefficient with the monotonicity of functions appeared in Lemma~\ref{lem G1} and Lemma~\ref{lem Gamma}, we verify that it is accurate.

As applications, we first give four different kinds of examples corresponding to the nonzero rank conditions of the main theorem. Then the finite rank commutator $\left[T_{e^{ik_1\theta }r^{k_1}},T_{e^{ik_2\theta }\varphi}\right]$ is considered, one of whose symbols is analytic or co-analytic. Also, the corresponding problem of the commutator $\left[T_{e^{ik_1\theta }r^{m_1}},T_{e^{ik_2\theta }r^{m_2}}\right]$ is studied.

If all similar questions are raised under the condition of generalized semicommutators  on $L_h^2$, we can also obtain the complete descriptions in Section 5. In fact,
using the results in Section 2 and in Section 4, it is relatively easy to get necessary conditions for generalized semicommutators to be finite rank, but a further analysis is needed to get corresponding results of generalized semicommutators, particularly to determine the range of each finite generalized semicommutator. Of cause, the finite rank problem of the semicommutator $\left(T_{e^{ik_1\theta} r^m},T_{e^{ik_2\theta }\varphi}\right]$, as an application of main theorem, is completely solved in this section.

In Section 6 we modify our arguments used in
previous sections to obtain corresponding results on the Bergman space. To end this paper we will show some connections between finite rank commutators (or generalized semicommutators) of quasihomogeneous Toeplitz
operators on the Bergman space and on the harmonic Bergman space.

\section{Rank equivalence of generalized semicommutators of Toeplitz Operators}

When dealing with the finite rank operators, the following notation is advantageous. For two nontrivial functions $\phi$ and $\psi$ in $L_h^2$, define the operator $\phi\otimes\psi$ of rank one to be
$$\left(\phi\otimes\psi\right)h=\langle h,\psi\rangle \phi, \ \ \ \ h\in L_h^2.$$
We list here some well known and easy properties of $\phi\otimes\psi$, which will be used later.
\begin{proposition}\label{prop rankone} For two nontrivial functions $\phi$ and $\psi$ in $L_h^2$, the following statements hold.
\begin{enumerate}
  \item[(a)] $Ran\left(\phi\otimes\psi\right)=\mathbb{C}\phi$ and $\ker\left(\phi\otimes\psi\right)=\left[\psi\right]^{\bot}$.
  \item[(b)] $\left[\phi\otimes \psi\right]^{*}=\psi\otimes \phi$. Here $\left[\phi\otimes \psi\right]^{*}$ denotes the adjoint operator of $\phi\otimes\psi$.
  \item[(c)] $\overline{\left(\phi\otimes \psi\right)(h)}=\left(\overline{\phi}\otimes \overline{\psi}\right)(\overline{h})$.
  \item[(d)] Every operator of finite rank $N\geq 1$ on $L_h^2$ has a canonical form: $\sum\limits_{j = 1}^N {\phi_j  \otimes \psi_j }$
for some linearly independent sets $\left\{\phi_1,\cdots,\phi_N\right\}$ and $\{\psi_1,\cdots,\psi_N\}$ in $L_h^2$.
\end{enumerate}
\end{proposition}

We are now ready to state and prove our result about rank equivalence of generalized semicommutators of Toeplitz operators with general symbols.

\begin{theorem}\label{thm eqvi} Let $f_{1}$, $f_{2}$ and $f$ be T-functions on $D$. If the generalized semicommutator $T_{f_{1}}T_{f_{2}}-T_{f}$ has finite rank $N$ on $L_h^2$, then the anti-ordered generalized semicommutator $T_{f_{2}}T_{f_{1}}-T_{f}$ also has finite rank $N$.
\end{theorem}
\begin{proof}Since $R_a$ is real-valued for each
$a\in {D}$, we note $\overline{Q(g)}=Q(\overline{g})$ for every $g\in L^2(D,dA)$, and hence
\begin{align*}
\overline{\left(T_{f_{2}}T_{f_{1}}-T_{f}\right)(h)}&=Q(\overline{f_2}Q(\overline{f_{1}} \overline{h}))-Q(\overline{f} \overline{h})\\
&=\left(T_{\overline{f_{2}}}T_{\overline{f_{1}}}-T_{\overline{f}}\right)(\overline{h})=\left(T_{f_{1}}T_{f_{2}}-T_{f}\right)^*(\overline{h})
\end{align*}
for any $h\in L_h^2$.
Hence $\left(T_{f_{2}}T_{f_{1}}-T_{f}\right)(h)=\overline{\left(T_{f_{1}}T_{f_{2}}-T_{f}\right)^*(\overline{h})}.$

Suppose $T_{f_{1}}T_{f_{2}}-T_{f}$ has finite rank $N$. If $N=0$, then it follows that $T_{f_{2}}T_{f_{1}}-T_{f}=0$, as desired. So assume $N\geq1$.
We write
\begin{equation}\label{R}
    T_{f_{1}}T_{f_{2}}-T_{f}=\sum\limits_{j = 1}^N {\phi_j  \otimes \psi_j }
\end{equation}
for some linearly independent sets $\left\{\phi_1,\cdots,\phi_N\right\}$ and $\{\psi_1,\cdots,\psi_N\}$ in $L_h^2$.
Then by Proposition~\ref{prop rankone}, we get
\begin{align*}
\left(T_{f_{2}}T_{f_{1}}-T_{f}\right)(h)&=\sum\limits_{j = 1}^N\overline{ \left[{\phi_j  \otimes \psi_j }\right]^*(\overline{h})}=\sum\limits_{j = 1}^N\left[{\overline{\psi_j}  \otimes \overline{\phi_j} }\right]({h}),
\end{align*}
which implies
\begin{equation}\label{S}
T_{f_{2}}T_{f_{1}}-T_{f}=\sum\limits_{j = 1}^N{\overline{\psi_j}  \otimes \overline{\phi_j} }.
\end{equation}
So $T_{f_{2}}T_{f_{1}}-T_{f}$ also has finite rank $N$, as desired.
\end{proof}

Theorem~\ref{thm eqvi} yields the following interesting result, showing the connection of finite rank commutator and corresponding generalized semicommutator of Toeplitz operators on $L_h^2$.

\begin{corollary}\label{cor relation} Let $f_{1}$, $f_{2}$ and $f$ be T-functions on $D$. If the generalized semicommutator $T_{f_{1}}T_{f_{2}}-T_{f}$ has finite rank $N$ on $L_h^2$,
then the commutator $\left[T_{f_{1}},T_{f_{2}}\right]$ also has finite rank, and the rank is less than or equal to $2N$.
\end{corollary}
\begin{proof}
It follows from \eqref{R} and \eqref{S} that
$$\left[T_{f_1},T_{f_2}\right]=\sum\limits_{j = 1}^N{\left[\phi_j  \otimes \psi_j-\overline{\psi_j}  \otimes \overline{\phi_j} \right]}$$
for some linearly independent sets $\left\{\phi_1,\cdots,\phi_N\right\}$ and $\left\{\psi_1,\cdots,\psi_N\right\}$ in $L_h^2$. Observe that the set $\left\{\phi_1,\cdots,\phi_N, \overline{\psi_1},\cdots,\overline{\psi_N}\right\}$ may be linearly dependent, and the result follows.
\end{proof}

As a direct consequence of Theorem~\ref{thm eqvi} and Corollary~\ref{cor relation}, one can get the following interesting
corollary. In fact, a particular case of it when both $f_{1}$ and $f_{2}$ are quasihomogeneous has been proved in \cite{DZ4}, using quite different methods from those that we use here.

\begin{corollary} Let $f_{1}$ and $f_{2}$ be two T-functions on $D$. If there exists a T-function $f$ such that $T_{f_{1}}T_{f_{2}}=T_{f}$ on $L_h^2$,
then $T_{f_{2}}T_{f_{1}}=T_{f}$, and hence $T_{f_{1}}T_{f_{2}}=T_{f_{2}}T_{f_{1}}$.
\end{corollary}

\section{Preliminary results}

We start this section with the concept of the Mellin transform.
For a radial function $f$ on $D$, the Mellin transform of $f$ is the function
$\widehat{f}$ defined by
$$\widehat{f}(z)=\int_0^1  {f  (r)r^{z - 1} dr}.$$ If $f \in
L^{1}([0,1],rdr)$, then
$\widehat{f }$ is well defined on the right half-plane $\left\{z:
\textrm{Re}z\geq 2\right\}$ and analytic on $\left\{z: \textrm{Re}z>
2\right\}$. For another radial function $g\in L^{1}([0,1],rdr)$, then their Mellin convolution is defined
by
$$(f\ast_{M}g)(r)=\int^1_{r}f\left(\frac{r}{t}\right)g(t)\frac{dt}{t},\;\;0\leq r<1.$$
The Mellin convolution theorem states that $\widehat{(f\ast_{M}g)}(z)=\widehat{f}(z)\widehat{g}(z)$.
Furthermore, it is known that $f\ast_{M}g$ is also in $L^{1}([0,1],rdr)$.

We shall often use the following lemma which is taken from \cite[Lemma 2.1]{DZ2}.

\begin{lemma}\label{lem calc} Let $k\in\mathbb{Z}$ and let $\varphi$ be a
radial T-function. Then on $L_h^2$, for each $n\in \mathbb{N}$ we have
\begin{align*}
&T_{ e^{ik\theta}\varphi }(z^{n})=\left\{ {\begin{array}{ll}
   {2(n+k+1)\widehat{\varphi }(2n+k+2)z^{n+k}}, & \text{if\; $n\geq -k$},  \\
   {2(-n-k+1)\widehat{\varphi }(-k+2)\overline{z}^{-n-k} },
   & \text{if\; $n<-k$;}  \\
\end{array}} \right.\\
&T_{ e^{ik\theta}\varphi }(\overline{z}^{n})=\left\{
{\begin{array}{ll}
   {2(n-k+1)\widehat{\varphi }(2n-k+2)\overline{z}^{n-k}}, & \text{if\; $n\geq k$},  \\
   {2(k-n+1)\widehat{\varphi }(k+2)z^{k-n} },
   & \text{if\; $n <k$.}
\end{array}} \right.
\end{align*}
\end{lemma}

A direct calculation gives the following essential lemma.
\begin{lemma}\label{lem constrela}
Let $k,\;l\in\mathbb{Z}$ and let $\varphi$ be a
radial T-function. Then on $L_h^2$,
$$T_{ e^{ik\theta}\varphi }\left(r^{|l|}e^{il\theta}\right)=\lambda_{k,l}\
r^{|k+l|}e^{i(k+l)\theta}$$ for some constant $\lambda_{k,l}$. Moreover,
\begin{equation}\label{lamre}
\left(|-l|+1\right)\lambda_{k,l}=\left(|l+k|+1\right)\lambda_{k,-l-k}.
\end{equation}
\end{lemma}
\begin{proof} By Lemma~\ref{lem calc}, it can easily be checked that
$T_{ e^{ik\theta}\varphi }\left(r^{|l|}e^{il\theta}\right)=\lambda_{k,l}\
r^{|k+l|}e^{i(k+l)\theta},$ where
\begin{equation}\label{laml}
 \lambda_{k,l}=\left\{ {\begin{array}{ll}
   {2(l+k+1)\widehat{\varphi }(2l+k+2)}, & \text{if\; $l\geq0$\; and\; $l\geq -k$, }  \\
   {2(-l-k+1)\widehat{\varphi }(-k+2)},
   & \text{if\; $l\geq0$\; and\; $l\leq-k$,}  \\
   {2(-l-k+1)\widehat{\varphi }(-2l-k+2)}, & \text{if\; $l\leq0$\; and\; $-l\geq k$},  \\
   {2(l+k+1)\widehat{\varphi }(k+2)},
   & \text{if\; $l\leq0$\; and\; $-l \leq k$.}
\end{array}} \right.
\end{equation}
Observe that
\begin{itemize}
\item $l\geq0$ and $l\geq -k$ $\Longleftrightarrow$ $-l-k\leq0$ and $l+k\geq k$;
\item $l\geq0$\; and\; $l\leq-k$ $\Longleftrightarrow$ $-l-k\geq0$ and $-l-k\leq -k$;
\item $l\leq0$\; and\; $-l\geq k$ $\Longleftrightarrow$ $-l-k\geq0$ and $-l-k\geq -k$;
\item $l\leq0$\; and\; $-l\leq k$ $\Longleftrightarrow$ $-l-k\leq0$ and $l+k\leq k$.
\end{itemize}
Then by \eqref{laml}, we get
\begin{equation}\label{laml2}
\lambda_{k,-l-k}=\left\{ {\begin{array}{ll}
   {2(l+1)\widehat{\varphi }(2l+k+2)}, & \text{if\; $l\geq0$\; and\; $l\geq -k$, } \\
   {2(l+1)\widehat{\varphi }(-k+2)},
   & \text{if\; $l\geq0$\; and\; $l\leq-k$,}  \\
   {2(-l+1)\widehat{\varphi }(-2l-k+2)}, & \text{if\; $l\leq0$\; and\; $-l\geq k$},  \\
   {2(-l+1)\widehat{\varphi }(k+2)},
   & \text{if\; $l\leq0$\; and\; $-l \leq k$.}
\end{array}} \right.
\end{equation}
Comparing \eqref{laml} with \eqref{laml2}, it is clear that \eqref{lamre} holds. This completes the proof.
\end{proof}

By complex and lengthy calculations, we proved the following technical lemma in \cite{DZ3}.
Using the critical equation \eqref{lamre},
now we can give a new and quite simple proof.

\begin{lemma}\label{lem sym}  Let $k_1,\;k_2\in\mathbb{Z}$, and let ${\varphi_{1}},\;{\varphi_{2}}$ be two radial T-functions. Then on $L_h^2$, for any $k\in \mathbb{Z}$ we have
\begin{align*}
&T_{e^{ik_1\theta} \varphi_{1}}T_{e^{ik_2\theta} \varphi_{2}}\left(r^{|k|}e^{ik\theta}\right)=T_{e^{ik_2\theta} \varphi_{2}}T_{e^{ik_1\theta} \varphi_{1}}\left(r^{|k|}e^{ik\theta}\right)\Longleftrightarrow\\
&T_{e^{ik_1\theta} \varphi_{1}}T_{e^{ik_2\theta} \varphi_{2}}\left(r^{|-k-k_1-k_2|}e^{i(-k-k_1-k_2)\theta}\right)=T_{e^{ik_2\theta} \varphi_{2}}T_{e^{ik_1\theta} \varphi_{1}}\left(r^{|-k-k_1-k_2|}e^{i(-k-k_1-k_2)\theta}\right).
\end{align*} Moreover, if $-k_1-k_2$ is even, then
$$T_{e^{ik_1\theta} \varphi_{1}}T_{e^{ik_2\theta} \varphi_{2}}\left(r^{|\frac{-k_1-k_2}{2}|}e^{\frac{i(-k_1-k_2)\theta}{2}}\right)\equiv
T_{e^{ik_2\theta} \varphi_{2}}T_{e^{ik_1\theta} \varphi_{1}}\left(r^{|\frac{-k_1-k_2}{2}|}e^{\frac{i(-k_1-k_2)\theta}{2}}\right).$$
\end{lemma}

\begin{proof}
It follows from Lemma~\ref{lem constrela} that
\begin{align*}
&T_{e^{ik_1\theta} \varphi_{1}}T_{e^{ik_2\theta} \varphi_{2}}\left(r^{|k|}e^{ik\theta}\right)=T_{e^{ik_2\theta} \varphi_{2}}T_{e^{ik_1\theta} \varphi_{1}}\left(r^{|k|}e^{ik\theta}\right)\Longleftrightarrow
\lambda_{k_2,k}\lambda_{k_1,k+k_2}=\lambda_{k_1,k}\lambda_{k_2,k+k_1}\\
&\Longleftrightarrow
\frac{\left(|-k|+1\right)\lambda_{k_2,k}\left(|-k-k_2|+1\right)\lambda_{k_1,k+k_2}}{|-k-k_2|+1}=
\frac{\left(|-k|+1\right)\lambda_{k_1,k}\left(|-k-k_1|+1\right)\lambda_{k_2,k+k_1}}{|-k-k_1|+1}.
\end{align*}
By \eqref{lamre}, a direct calculation shows that the above equation is equivalent to
$$
\lambda_{k_2,-k-k_1-k_2}\lambda_{k_1,-k-k_1}=\lambda_{k_1,-k-k_1-k_2}\lambda_{k_2,-k-k_2}.
$$
Using Lemma~\ref{lem constrela} again, then the desired result follows immediately.
Moreover, if $k_1-k_2$ is even, then using \eqref{lamre} again, we get
\begin{align*}
    \lambda_{k_2,\frac{-k_1-k_2}{2}}\lambda_{k_1,\frac{-k_1+k_2}{2}}&=
\frac{\left(\left|\frac{k_1+k_2}{2}\right|+1\right)\lambda_{k_2,\frac{-k_1-k_2}{2}}
    \left(\left|\frac{k_1-k_2}{2}\right|+1\right)\lambda_{k_1,\frac{-k_1+k_2}{2}}}{\left(\left|\frac{k_1+k_2}{2}\right|+1\right)\left(\left|\frac{k_1-k_2}{2}\right|+1\right)}\\
    &=\lambda_{k_1,\frac{-k_1-k_2}{2}}\lambda_{k_2,\frac{k_1-k_2}{2}},
\end{align*}
which completes the proof.
\end{proof}

Recall that $\frac{-k_1-k_2}{2}$ is called the symmetry number of the commutator $\left[T_{e^{ik_1\theta} \varphi_{1}},T_{e^{ik_2\theta }\varphi_2}\right]$ in \cite{DZ3}. Since $\left\{\sqrt{|l|+1}r^{|l|}e^{il\theta}\right\}_{l\in\mathbb{Z}}$ is an
orthonormal basis for $L_h^2$, it will greatly simplify the analysis if we use this symmetry property to determine the range of the finite rank commutator two quasihomogeneous Toeplitz operators on $L^{2}_{h}$.

Lemma~\ref{lem constrela} implies that $T_{ e^{ik\theta}\varphi }$ acts on the elements of the orthogonal basis of $L_h^2$
as a weight shift operator. From this we can give the following lemma which shows the relation between the range and the canonical form of finite rank quasihomogeneous Toeplitz operators on $L^{2}_{h}$.

\begin{lemma}\label{lem comrank}  Let $T_{e^{ip_i\theta} \varphi_{i}},\ T_{e^{iq_i\theta }\psi_i}$ be quasihomogeneous Toeplitz operators on $L^{2}_{h}$ and $p_i+q_i=k$ for any $i=1,\cdots,i_0$. Then
\begin{equation}\label{ranforeq1}
Ran\left(\sum\limits_{i = 1}^{i_0}T_{e^{ip_i\theta} \varphi_{i}}T_{e^{iq_i\theta }\psi_i}\right)=
\emph{span}\left\{r^{|k_{(1)}|}e^{ik_{(1)}\theta}, r^{|k_{(2)}|}e^{ik_{(2)}\theta},\cdots, r^{|k_{(N)}|}e^{ik_{(N)}\theta}\right\}\end{equation}
if and only if
\begin{equation}\label{ranforeq2}
\sum\limits_{i = 1}^{i_0}T_{e^{ip_i\theta} \varphi_{i}}T_{e^{iq_i\theta }\psi_i}=\sum\limits_{j = 1}^{N} {C_{k_{(j)}} \left(r^{|k_{(j)}|}e^{ik_{(j)}\theta}\right)  \otimes \left(r^{|k_{(j)}-k|}e^{i\left(k_{(j)}-k\right)\theta}\right) }
\end{equation}
for some nonzero constant $C_{k_{(j)}}$.
\end{lemma}
\begin{proof} First suppose \eqref{ranforeq1} holds.
Notice that $p_i+q_i=k$, then by Lemma~\ref{lem constrela}, for any $l\in\mathbb{Z}$
$$\sum\limits_{i = 1}^{i_0}T_{e^{ip_i\theta} \varphi_{i}}T_{e^{iq_i\theta }\psi_i}\left(r^{|l|}e^{il\theta}\right)
=\left\{
{\begin{array}{ll}
   {0}, & \text{if\; $l\neq k_{(j)}-k$},  \\
   {C_{k_{(j)}}^{\prime}r^{|k_{(j)}|}e^{ik_{(j)}\theta}},
   & \text{if\; $l=k_{(j)}-k$},
\end{array}} \right.
$$ for some nonzero constant $C_{k_{(j)}}^{\prime}$,
where $j\in\{1,2,\cdots,N\}$. On the other hand,
\begin{align*}
&\left[\sum\limits_{j = 1}^{N} {\left(r^{|k_{(j)}|}e^{ik_{(j)}\theta}\right)  \otimes \left(r^{|k_{(j)}-k|}e^{i\left(k_{(j)}-k\right)\theta}\right) }\right]\left(r^{|l|}e^{il\theta}\right)\\
&=\sum\limits_{j = 1}^{N} {\left\langle r^{|l|}e^{il\theta}, r^{|k_{(j)}-k|}e^{i\left(k_{(j)}-k\right)\theta}\right\rangle r^{|k_{(j)}|}e^{ik_{(j)}\theta} }=\left\{
{\begin{array}{ll}
   {0}, & \text{if\; $l\neq k_{(j)}-k$},  \\
   {\frac{1}{{|k_{(j)}-k|+1}}r^{|k_{(j)}|}e^{ik_{(j)}\theta}},
   & \text{if\; $l=k_{(j)}-k$}.
\end{array}} \right.
\end{align*}
Thus, \eqref{ranforeq2} holds
for nonzero constant $C_{k_{(j)}}=\left(|k_{(j)}-k|+1\right)C_{k_{(j)}}^{\prime}$.

The converse implication is a direct consequence of Proposition~\ref{prop rankone}, finishing the proof.
\end{proof}

The following two lemmas play critical roles in next two sections. Before we state our results, we need to introduce concepts of the digamma function and the polygamma function.
The digamma function $\psi(z)$ is defined as the logarithmic derivative of the gamma function $\Gamma(z)$, that is:
$$\psi(z)=\frac{d}{dz}\log\Gamma(z)=\frac{\Gamma'(z)}{\Gamma(z)}.$$
The derivatives $\psi^{\prime},
\psi^{\prime\prime}, \psi^{\prime\prime\prime},\ldots$ are known as the
tri-, tetra-, pentagamma functions or, generally, the polygamma functions.
The polygamma function may be represented as
$$\psi^{(m)}(z)=(-1)^{m+1}\int_0^{+\infty}\frac{t^m}{1-e^{-t}}e^{-zt}dt$$
for $\textrm{Re} z >0$ and $m>0$. Obviously, $\psi^{\prime\prime}(s)<0$ for $s\in (0,+\infty)$, and hence $\psi^{\prime}$ is strictly monotone decreasing on $(0,+\infty)$. We refer to \cite[p.260]{AS} for the properties of these functions.

\begin{lemma}\label{lem G1} Let $b>0$, and let
$$
F(x)=
\frac{\Gamma\left(x\right)\Gamma\left(x+b-a\right)}{\Gamma\left(x-a\right)\Gamma\left(x+b\right)}.
$$
\begin{enumerate}
\item[(a)] If $a>0$, then $F(x)$ is strictly monotone increasing on $(a,+\infty)$.
\item[(b)]  If $a<0$, then $F(x)$ is strictly monotone decreasing on $(0,+\infty)$.
\end{enumerate}
\end{lemma}
\begin {proof} First suppose $a>0$.
Define
\[
g(x):=\psi(x)+\psi(x+b-a)-\psi(x-a)-\psi(x+b)
\]
for $x\in (a,+\infty)$, where $\psi$ is the digamma function.

Now, we denote $h(x)=\psi(x)-\psi(x-a)$ for $x>a$. Since $a>0$, we get
$$h^{\prime}(x)=\psi^{\prime}(x)-\psi^{\prime}(x-a)<0,$$
and hence
$h(x)>h(x+b),\ \forall\ b>0.$
Consequently, $g(x)=h(x)-h(x+b)>0$ for all $x\in(a,+\infty)$. But note that
$
\frac{d}{dx}\log F(x)=g(x).
$
So $F(x)$ is strictly monotone increasing  on $(a,+\infty)$.

The proof of (b) is similar. This completes the proof.
\end{proof}

\begin{lemma}\label{lem Gamma} Let $b>0$, and let
\begin{equation}\label{G}
    G(x)=\frac{\Gamma\left(-x+a\right)\Gamma\left(-x+b+1\right)}
{\Gamma\left(-x+1\right)\Gamma\left(-x+b+a\right)}.
\end{equation}
\begin{enumerate}
\item[(a)] If $a\in(0,1)$, then $G(x)$ is strictly monotone increasing on $(-\infty,a)$.
\item[(b)] If $a\in(1,+\infty)$, then $G(x)$ is strictly monotone decreasing on $(-\infty,1)$.
\end{enumerate}
\end{lemma}
\begin{proof} First suppose $a\in(0,1)$.
Define
\[
g(x):=\psi(-x+a)+\psi(-x+b+1)-\psi(-x+1)-\psi(-x+b+a )
\]
for $x\in (-\infty,a)$, where $\psi$ is the digamma function.

Fix $x$ and denote $h(y)=\psi(-x+y+1)-\psi(-x+y+a)$ for $y\geq 0$. Since $a<1$, we get
$$h^{\prime}(y)=\psi^{\prime}(-x+y+1)-\psi^{\prime}(-x+y+a)<0,$$
Recall that $b>0$ and therefore
$h(b)<h(0)$,
which implies that $$g(x)=h(b)-h(0)<0$$ for all $x\in(-\infty,a)$. But note that
$
\frac{d}{dx}\log G(x)=-g(x).
$
So $G(x)$ is strictly monotone increasing on $(-\infty,a)$.

The proof of (b) is similar. This completes the proof.
\end{proof}

Inspired by \cite{CR}, we give the following lemma, which gives the sufficient and necessary conditions for the function appeared in Section 4 to be a T-function.
\begin{lemma}\label{lem vak}  Let $k_1,\;k_2\in\mathbb{Z}$ such that $k_1>0$, and let $m\in \mathbb{R}$, $m\geq -1$. Assume $\varphi(r)$ is a radial function on $D$ such that
\begin{equation}\label{gk}
\widehat{\varphi}(z)=\frac{\Gamma\left(\frac{z+k_2}{2k_1}\right)\Gamma\left(\frac{z+m+k_1-k_2}{2k_1}\right)}{\Gamma\left(\frac{z+2k_1-k_2}{2k_1}\right)
\Gamma\left(\frac{z+m+k_1+k_2}{2k_1}\right)}.
\end{equation}
Then $\varphi(r)$ is a T-function if and only if one of the following conditions holds:
\begin{enumerate}
\item[(1)]$k_2\leq-2$ and $m+k_1=0$.
\item[(2)]$-2<k_2<m+k_1+2$.
\item[(3)]$k_2\geq m+k_1+2$ and $m=(2n+1)k_1$ for some $n\in\mathbb{N}$.
\end{enumerate}
\end{lemma}
\begin{proof} First assume $k_2\leq-2$. If $\varphi(r)$ is a T-function on $D$, then $\varphi(r)\in L^1(D, dA)$, and hence $\widehat{\varphi}(z)$ must be well defined on $\left\{z:\textrm{Re}z\geq 2\right\}$.
Therefore, either $\frac{2k_1-2k_2}{2k_1}$ or $\frac{m+k_1}{2k_1}$ is an integer. Otherwise, $-k_2$ must be a pole of $\widehat{\varphi}(z)$ in $\left\{z:\textrm{Re}z\geq 2\right\}$, which leads to a contradiction.

If $\frac{2k_1-2k_2}{2k_1}$ is an integer, then we can suppose $-k_2=(n+1)k_1$ for some $n\in \mathbb{N}$, since $k_1>0$ and $k_2\leq-2$. Thus, \eqref{gk} becomes
\begin{align*}
\widehat{\varphi}(z)&=\frac{\Gamma\left(\frac{z+k_2}{2k_1}\right)\Gamma\left(\frac{z+m+k_1+k_2}{2k_1}+n+1\right)}
{\Gamma\left(\frac{z+k_2}{2k_1}+n+2\right)\Gamma\left(\frac{z+m+k_1+k_2}{2k_1}\right)}\\
&=\frac{\frac{z+m+k_1+k_2}{2k_1}\left(\frac{z+m+k_1+k_2}{2k_1}+1\right)\cdots\left(\frac{z+m+k_1+k_2}{2k_1}+n\right)}
{\frac{z+k_2}{2k_1}\left(\frac{z+k_2}{2k_1}+1\right)\cdots\left(\frac{z+k_2}{2k_1}+n+1\right)  }.
\end{align*}
So $\widehat{\varphi}(z)$ is a proper fraction and can be written as a sum of partial fractions $\sum\limits_{j=0}^{n+1}{\frac{A_j}{z+k_2+2k_1j}}$.
Since $k_2\leq-2$ and $\widehat{\varphi}(z)$ is well defined on $\left\{z:\textrm{Re}z\geq 2\right\}$, it follows that $A_0=0$, which implies
$$\frac{m+k_1+k_2}{2k_1}+i=\frac{k_2}{2k_1}$$ must hold for some $i\in\left\{0,1,\cdots,n\right\}$. Since $m\geq-1$ and $k_1>0$, we have $m+k_1=0$.

If $\frac{m+k_1}{2k_1}$ is an integer, then we can suppose $m=(2n-1)k_1$ for some $n\in \mathbb{N}$. If $m+k_1\neq0$, then $n\geq1$. Thus, \eqref{gk} becomes
\begin{align*}
\widehat{\varphi}(z)&
=\frac{\left(\frac{z-k_2}{2k_1}+1\right)\left(\frac{z-k_2}{2k_1}+2\right)\cdots\left(\frac{z-k_2}{2k_1}+n-1\right)}
{\frac{z+k_2}{2k_1}\left(\frac{z+k_2}{2k_1}+1\right)\cdots\left(\frac{z+k_2}{2k_1}+n-1\right)}=\sum\limits_{j=0}^{n-1}{\frac{A_j}{z+k_2+2k_1j}}.
\end{align*}
Notice that $\frac{-k_2}{2k_1}+i>\frac{k_2}{2k_1}$ for all $i\in\left\{0,1,\cdots,n-1\right\}$, so $A_0\neq0$, which contradicts the fact that $\widehat{\varphi}(z)$ is well defined on $\left\{z:\textrm{Re}z\geq 2\right\}$.
Thus $m+k_1=0$.

Conversely, assume $m+k_1=0$. Then \eqref{gk} becomes $\widehat{\varphi}(z)=\frac{\Gamma\left(\frac{z+k_2}{2k_1}\right)\Gamma\left(\frac{z-k_2}{2k_1}\right)}{\Gamma\left(\frac{z+2k_1-k_2}{2k_1}\right)
\Gamma\left(\frac{z+k_2}{2k_1}\right)}=\frac{2k_1}{z-k_2},$ which yields $\varphi(r)=2k_1r^{-k_2}$, and hence $\varphi(r)$ is a T-function.

Next, assume $-2<k_2<m+k_1+2$. Then it is easy to check that $\widehat{\varphi}(z)$ is well defined on $\left\{z:\textrm{Re}z\geq 2\right\}$ in this particular situation. Furthermore, if $\widehat{\varphi}(z)$ is a rational function, then it is obvious that $\varphi(r)$ is a T-function. Now, we suppose $\widehat{\varphi}(z)$ is a irrational function.
Then using the same reasoning as in the proof of \cite[Theorem 4]{CR}, one can show that $\varphi(r)\in L^1(D, dA)$ and $\varphi(r)$ is "nearly bounded", which imply $\varphi(r)$ is a T-function.

Finally, assume $k_2\geq m+k_1+2$. Suppose that $\varphi(r)$ is a T-function on $D$. Similar as before, then either $\frac{-2k_2}{2k_1}$
or $\frac{m-k_1}{2k_1}$ is an integer. If $\frac{-2k_2}{2k_1}$ is an integer, then we can suppose $k_2=(n+2)k_1$ for some $n\in \mathbb{N}$, since $k_2>k_1>0$. Thus by a simple calculation, we obtain from \eqref{gk} that
\begin{align*}
\widehat{\varphi}(z)&=\frac{\frac{z+2k_1-k_2}{2k_1}\left(\frac{z+2k_1-k_2}{2k_1}+1\right)\cdots\left(\frac{z+2k_1-k_2}{2k_1}+n\right)}
{\frac{z+m+k_1-k_2}{2k_1}\left(\frac{z+m+k_1-k_2}{2k_1}+1\right)\cdots\left(\frac{z+m+k_1-k_2}{2k_1}+n+1\right)  }=\sum\limits_{j=0}^{n+1}{\frac{A_j}{z+m+k_1-k_2+2k_1j}}.
\end{align*}
Since $m+k_1-k_2\leq-2$, it follows that $A_0=0$, which implies
$$\frac{2k_1-k_2}{2k_1}+i=\frac{m+k_1-k_2}{2k_1}$$ must hold for some $i\in\left\{0,1,\cdots,n\right\}$, and hence $m=(2i+1)k_1.$

On the other hand, if $\frac{m-k_1}{2k_1}$ is an integer, then we can suppose $m=(2n+1)k_1$ for some integer $n\geq -1$, since $m\geq-k_1$.
However, if $n=-1$, then $m+k_1=0$ and \eqref{gk} becomes
$$\widehat{\varphi}(z)=\frac{\Gamma\left(\frac{z-k_2}{2k_1}\right)}{\Gamma\left(\frac{z-k_2}{2k_1}+1\right)}=\frac{2k_1}{z-k_2},$$ which contradicts the fact that $\widehat{\varphi}(z)$ is well defined on $\left\{z:\textrm{Re}z\geq 2\right\}$
since $k_2\geq m+k_1+2=2$.

Conversely, assume $m=(2n+1)k_1$ for some $n\in\mathbb{N}$. Then \eqref{gk} becomes
\begin{align*}
\widehat{\varphi}(z)&=\frac{\Gamma\left(\frac{z+k_2}{2k_1}\right)\Gamma\left(\frac{z-k_2}{2k_1}+n+1\right)}{\Gamma\left(\frac{z-k_2}{2k_1}+1\right)
\Gamma\left(\frac{z+k_2}{2k_1}+n+1\right)}=\frac{\left(\frac{z-k_2}{2k_1}+1\right)\left(\frac{z-k_2}{2k_1}+2\right)\cdots\left(\frac{z-k_2}{2k_1}+n\right)}
{\frac{z+k_2}{2k_1}\left(\frac{z+k_2}{2k_1}+1\right)\cdots\left(\frac{z+k_2}{2k_1}+n\right)  }.
\end{align*}
Clearly, $\varphi(r)$ is a T-function. This completes the proof.
\end{proof}

\section{Ranks of commutators on $L_h^2$}

In this section,
we will discuss the finite rank commutators of quasihomogeneous Toeplitz
operators on $L_h^2$. Throughout the rest part of this paper, we will use the following notations for brevity. For $k_1,\;k_2\in\mathbb{Z}$, we
denote $$N_1=\max\left\{0,-k_1,-k_2,-k_1-k_2\right\}.$$
Then we write
$${\Lambda_1}=\left\{k\in{\mathbb{Z}}: -N_1+1\leq k\leq N_1+k_1+k_2-1, k\neq \frac{k_1+k_2}{2}\right\},$$
and $${\Lambda_{\frac{1}{2}}}=\left\{k\in{\mathbb{Z}}: \frac{k_1+k_2}{2}< k \leq N_1+k_1+k_2-1\right\}.$$
First we characterize the finite rank commutators of two Toeplitz
operators with general quasihomogeneous symbols.
\begin{proposition}\label{prop commutemax}  Let $k_1,\;k_2\in\mathbb{Z}$, and let ${\varphi_{1}}$ and ${\varphi_{2}}$ be two radial T-functions. Then the commutator $\left[T_{e^{ik_1\theta} \varphi_{1}},T_{e^{ik_2\theta }\varphi_2}\right]$ has finite rank on $L_h^2$ if and only if
\begin{align}\label{cmax}
    &2(n+k_2+1)\widehat{\varphi_1
}(2n+k_1+2k_2+2)\widehat{\varphi_2
}(2n+k_2+2)\nonumber\\
&=2(n+k_1+1)\widehat{\varphi_1
}(2n+k_1+2)\widehat{\varphi_2
}(2n+2k_1+k_2+2)
\end{align}
holds for any natural number $n\geq N_1$. In this case,
\begin{align}\label{rangec}
Ran\left(\left[T_{e^{ik_1\theta} \varphi_{1}},T_{e^{ik_2\theta }\varphi_2}\right]\right) &=
\emph{span}\left\{r^{|k_{(1)}|}e^{ik_{(1)}\theta}, r^{|k_{(2)}|}e^{ik_{(2)}\theta},\cdots, r^{|k_{(N)}|}e^{ik_{(N)}\theta},\right.\\
& \left.r^{|-k_{(1)}+k_1+k_2|}e^{i\left(-k_{(1)}+k_1+k_2\right)\theta},\cdots, r^{|-k_{(N)}+k_1+k_2|}e^{i\left(-k_{(N)}+k_1+k_2\right)\theta}\right\}\nonumber
\end{align}
for some $k_{(1)},\cdots, k_{(N)}\in \Lambda_{\frac{1}{2}}$,
and
\begin{align}\label{pair}
     \left[T_{e^{ik_1\theta} \varphi_{1}},T_{e^{ik_2\theta }\varphi_2}\right]&=\sum\limits_{j = 1}^{N} {C_{k_{(j)}}\left[ \left(r^{|k_{(j)}|}e^{ik_{(j)}\theta}\right)  \otimes \left(r^{|k_{(j)}-k_1-k_2|}e^{i\left(k_{(j)}-k_1-k_2\right)\theta}\right)\right.}\nonumber\\
    &\ \ \ \ \ \ \ \ \ \ \ \ \ \ \ {\left.- \left(r^{|-k_{(j)}+k_1+k_2|}e^{i\left(-k_{(j)}+k_1+k_2\right)\theta}\right)\otimes \left(r^{|-k_{(j)}|}e^{-ik_{(j)}\theta}\right) \right]}
\end{align}
for some nonzero constant $C_{k_{(j)}}$. Furthermore,
$rank\left(\left[T_{e^{ik_1\theta} \varphi_{1}},T_{e^{ik_2\theta }\varphi_2}\right]\right)\leq|k_1|+|k_2|-1$ when $k_1+k_2$ is odd, and $rank\left(\left[T_{e^{ik_1\theta} \varphi_{1}},T_{e^{ik_2\theta }\varphi_2}\right]\right)\leq\max\left\{0,|k_1|+|k_2|-2\right\}$ when $k_1+k_2$ is even.
\end{proposition}
\begin{proof}
For any natural number $n\geq N_1$, by Lemma~\ref{lem calc} we get
\begin{align}\label{cl1}
    &\left[T_{e^{ik_1\theta} \varphi_1},T_{e^{ik_2\theta }\varphi_2}\right](z^n)\nonumber\\
    &=2(n+k_1+k_2+1)\left[2(n+k_2+1)\widehat{\varphi_1}(2n+k_1+2k_2+2)\widehat{\varphi_2}(2n+k_2+2)\right.\nonumber\\
    &\ \ \ \left.-2(n+k_1+1)\widehat{\varphi_1 }(2n+k_1+2)\widehat{\varphi_2}(2n+2k_1+k_2+2)\right]z^{n+k_1+k_2}.
\end{align}
Suppose that $\left[T_{e^{ik_1\theta} \varphi_1},T_{e^{ik_2\theta }\varphi_2}\right]$ has finite rank. Then the set $\left\{\left[T_{e^{ik_1\theta} \varphi_1},T_{e^{ik_2\theta }\varphi_2}\right](z^n):n\geq N_1\right\}$ contains finite
linearly independent vectors, and hence there exists some natural number $N_0\geq N_1$ such that
$$\left[T_{e^{ik_1\theta} \varphi_1},T_{e^{ik_2\theta }\varphi_2}\right](z^n)=0,\  \forall \ n\geq N_0.$$
Consequently, it follows from \eqref{cl1} that \eqref{cmax} holds for any $ n\geq N_0$.
According to the basic properties of Mellin transform, \eqref{cmax} can be written as
\begin{align*}
    &\widehat{\left(r^{2k_1+2N_1}\right)}(2n-2N_1+2)\widehat{\left(r^{k_1+2k_2+2N_1}\varphi_1\right)}(2n-2N_1+2)\widehat{\left(r^{k_2+2N_1}\varphi_2\right)}(2n-2N_1+2)\\
    &=\widehat{\left(r^{2k_2+2N_1}\right)}(2n-2N_1+2)\widehat{\left(r^{k_1+2N_1}\varphi_1\right)}(2n-2N_1+2)\widehat{\left(r^{2k_1+k_2+2N_1}\varphi_2\right)}(2n-2N_1+2).
\end{align*}
Since the sequence $\left(2n-2N_1+2\right)_{n\geq N_0}$ is arithmetic, then as in the proof of \cite[Theorem 6]{CuL}, we can get the above equation holds for any $n\geq N_1$, and hence \eqref{cmax} holds for any $ n\geq N_1$.

Conversely, assume \eqref{cmax} holds for any $ n\geq N_1$. Then by \eqref{cl1}, we get
$$\left[T_{e^{ik_1\theta} \varphi_1},T_{e^{ik_2\theta }\varphi_2}\right](z^n)=0,\  \forall \ n\geq N_1.$$
Furthermore, it follows from Lemma~\ref{lem sym} that
$$\left[T_{e^{ik_1\theta} \varphi_1},T_{e^{ik_2\theta }\varphi_2}\right](\overline{z}^{n+k_1+k_2})=0,\ \forall\  n\geq N_1,$$
and if $-k_1-k_2$ is even, then $$\left[T_{e^{ik_1\theta} \varphi_1},T_{e^{ik_2\theta }\varphi_2}\right]\left(r^{|\frac{-k_1-k_2}{2}|}e^{\frac{i(-k_1-k_2)\theta}{2}}\right)=0.$$
Therefore,
\begin{align*}
    &Ran\left(\left[T_{e^{ik_1\theta}\varphi_1 },T_{e^{ik_2\theta }\varphi_2}\right]\right)\\
    &=\emph{span}\left\{\left[T_{e^{ik_1\theta}\varphi_1 },T_{e^{ik_2\theta }\varphi_2}\right]\left(r^{|l|}e^{il\theta}\right):l\in{\mathbb{Z}}, -N_1-k_1-k_2+1\leq l\leq N_1-1, l\neq \frac{-k_1-k_2}{2}\right\}.
    \end{align*}
Now, using Lemma~\ref{lem constrela} and the definition of $\Lambda_1$, it yields
\begin{align}\label{subrange}
    Ran\left(\left[T_{e^{ik_1\theta}\varphi_1 },T_{e^{ik_2\theta }\varphi_2}\right]\right)
    &\subseteq \emph{span}\left\{r^{|k|}e^{ik\theta}:k\in{{\Lambda_1}}\right\},
\end{align}
which shows that the commutator $\left[T_{e^{ik_1\theta} \varphi_{1}},T_{e^{ik_2\theta }\varphi_2}\right]$ has finite rank.
Moreover, it follows from Lemma~\ref{lem sym} that,
if $r^{|k_{(j)}|}e^{ik_{(j)}\theta}\in Ran\left(\left[T_{e^{ik_1\theta} \varphi_{1}},T_{e^{ik_2\theta }\varphi_2}\right]\right)$ for some $k_{(j)}\in \Lambda_1 $, then $$r^{|-k_{(j)}+k_1+k_2|}e^{i(-k_{(j)}+k_1+k_2)\theta}\in Ran\left(\left[T_{e^{ik_1\theta} \varphi_{1}},T_{e^{ik_2\theta }\varphi_2}\right]\right).$$
Therefore, the monomials in $Ran\left(\left[T_{e^{ik_1\theta} \varphi_{1}},T_{e^{ik_2\theta }\varphi_2}\right]\right)$ must always appear in pairs. Notice that the set $\Lambda_1$ is also symmetric, and hence
\eqref{rangec} must hold for some $k_{(1)},\cdots, k_{(N)}\in \Lambda_{\frac{1}{2}}$.

Next, we want to determine the canonical form of the finite rank commutator $\left[T_{e^{ik_1\theta}\varphi_1 },T_{e^{ik_2\theta }\varphi_2}\right]$.
Obviously, by Lemma~\ref{lem comrank}, it follows from \eqref{rangec} that
 \begin{align}\label{decompose}
     \left[T_{e^{ik_1\theta} \varphi_{1}},T_{e^{ik_2\theta }\varphi_2}\right]
    &=\sum\limits_{j = 1}^{N} {C_{k_{(j)}} \left(r^{|k_{(j)}|}e^{ik_{(j)}\theta}\right)  \otimes \left(r^{|k_{(j)}-k_1-k_2|}e^{i\left(k_{(j)}-k_1-k_2\right)\theta}\right)}\nonumber\\
    &+C_{-k_{(j)}+k_1+k_2}\left(r^{|-k_{(j)}+k_1+k_2|}e^{i\left(-k_{(j)}+k_1+k_2\right)\theta}\right)\otimes \left(r^{|-k_{(j)}|}e^{-ik_{(j)}\theta}\right)
\end{align}
for nonzero constant $C_{k_{(j)}}=\left(|k_{(j)}-k_1-k_2|+1\right)C_{k_{(j)}}^{\prime}$,
where $C_{k_{(j)}}^{\prime}$ is the coefficient satisfying $$\left[T_{e^{ik_1\theta}\varphi_1 },T_{e^{ik_2\theta }\varphi_2}\right]\left(r^{|k_{(j)}-k_1-k_2|}e^{i\left(k_{(j)}-k_1-k_2\right)\theta}\right)=C_{k_{(j)}}^{\prime}r^{|k_{(j)}|}e^{ik_{(j)}\theta},$$ and nonzero  constant $C_{-k_{(j)}+k_1+k_2}=\left(|-k_{(j)}|+1\right)C_{-k_{(j)}+k_1+k_2}^{\prime}$, where $C_{-k_{(j)}+k_1+k_2}^{\prime}$ satisfies $$\left[T_{e^{ik_1\theta}\varphi_1 },T_{e^{ik_2\theta }\varphi_2}\right]\left(r^{|-k_{(j)}|}e^{-ik_{(j)}\theta}\right)=C_{-k_{(j)}+k_1+k_2}^{\prime}r^{|-k_{(j)}+k_1+k_2|}e^{i\left(-k_{(j)}+k_1+k_2\right)\theta}.$$
According to Lemma~\ref{lem constrela}, it follows
$C_{k_{(j)}}^{\prime}=\lambda_{k_{2},k_{(j)}-k_1-k_2}\ \lambda_{k_{1},k_{(j)}-k_1}-\lambda_{k_{1},k_{(j)}-k_1-k_2}\ \lambda_{k_{2},k_{(j)}-k_2}$
and
$C_{-k_{(j)}+k_1+k_2}^{\prime}=\lambda_{k_{2},-k_{(j)}}\ \lambda_{k_{1},-k_{(j)}+k_2}-\lambda_{k_{1},-k_{(j)}}\ \lambda_{k_{2},-k_{(j)}+k_1}.$
Then by \eqref{lamre}, it can easily be checked that
$$\left(|k_{(j)}-k_1-k_2|+1\right)\lambda_{k_{2},k_{(j)}-k_1-k_2}\ \lambda_{k_{1},k_{(j)}-k_1}=\left(|-k_{(j)}|+1\right)\lambda_{k_{1},-k_{(j)}}\ \lambda_{k_{2},-k_{(j)}+k_1}$$
and $$\left(|k_{(j)}-k_1-k_2|+1\right)\lambda_{k_{1},k_{(j)}-k_1-k_2}\ \lambda_{k_{2},k_{(j)}-k_2}=\left(|-k_{(j)}|+1\right)\lambda_{k_{2},-k_{(j)}}\ \lambda_{k_{1},-k_{(j)}+k_2},$$
which implies
$$\left(|k_{(j)}-k_1-k_2|+1\right)C_{k_{(j)}}^{\prime}=-\left(|-k_{(j)}|+1\right)C_{-k_{(j)}+k_1+k_2}^{\prime},$$
and consequently,
$C_{k_{(j)}}=-C_{-k_{(j)}+k_1+k_2}$.
Then it follows from \eqref{decompose} that \eqref{pair} holds.

Finally, observe that $$2N_1+(k_1+k_2)=|k_1|+|k_2|,$$ so \eqref{subrange} implies that $rank\left(\left[T_{e^{ik_1\theta} \varphi_{1}},T_{e^{ik_2\theta }\varphi_2}\right]\right)\leq |k_1|+|k_2|-1$ when $k_1+k_2$ is odd, and $rank\left(\left[T_{e^{ik_1\theta} \varphi_{1}},T_{e^{ik_2\theta }\varphi_2}\right]\right)\leq \max\left\{0,|k_1|+|k_2|-2\right\}$ when $k_1+k_2$ is even, as desired.
\end{proof}

\begin{remark}
\begin{enumerate}
\item It is obvious that the finite rank commutator $\left[T_{e^{ik_1\theta} \varphi_{1}},T_{e^{ik_2\theta }\varphi_2}\right]$ can't have an odd rank, which corresponds to the result of \cite{CKL}. Moreover, \eqref{pair} implies that the canonical form of finite rank commutator of quasihomogeneous Toeplitz operators must exist always in pairs.
\item One can also obtain further result by applying the adjoint operator. Notice that $$\left[T_{e^{ik_1\theta} \varphi_{1}},T_{e^{ik_2\theta }\varphi_2}\right]^{\ast}=-\left[T_{e^{-ik_1\theta} \overline{\varphi_{1}}},T_{{e^{-ik_2\theta }\overline\varphi_2}}\right],$$ then it follows from \eqref{pair} and Proposition~\ref{prop rankone} that
\begin{align*}
\left[T_{e^{-ik_1\theta} \overline{\varphi_{1}}},T_{{e^{-ik_2\theta }\overline\varphi_2}}\right]&=\sum\limits_{j = 1}^{N} \overline{C_{k_{(j)}}}\left[ \left(r^{|-k_{(j)}|}e^{-ik_{(j)}\theta}\right)  \otimes \left(r^{|-k_{(j)}+k_1+k_2|}e^{i\left(-k_{(j)}+k_1+k_2\right)\theta}\right)\right.\\
    &\ \ \ \ \ \ \ \ \ \ \ \ \ \ \ \ \ {\left.- \left(r^{|k_{(j)}-k_1-k_2|}e^{i\left(k_{(j)}-k_1-k_2\right)\theta}\right)\otimes \left(r^{|k_{(j)}|}e^{ik_{(j)}\theta}\right) \right]}.
\end{align*}
\end{enumerate}
\end{remark}

The following corollary gives a complete description of the finite rank commutators of
quasihomogeneous Toeplitz operators with some special
degrees.

\begin{corollary}\label{cor cradial}  Let $k\in\mathbb{Z}$, and let ${\varphi_{1}},\;{\varphi_{2}}$ be two nonzero radial T-functions. Then on $L_h^2$, the following statements hold.
\begin{enumerate}
\item[(a)] The commutator $\left[T_{\varphi_{1}},T_{e^{ik\theta }\varphi_2}\right]$ has finite rank if and only if $k=0$ or ${\varphi_{1}}$ is a constant.
\item[(b)] If $k\neq0$, then the commutator $\left[T_{e^{ik\theta }\varphi_{1}},T_{e^{ik\theta }\varphi_2}\right]$ has finite rank if and only if $\varphi_{1}=C\varphi_{2}$ for some constant $C$.
\item[(c)] If $k\neq0$, then the commutator $\left[T_{e^{ik\theta }\varphi_{1}},T_{e^{-ik\theta }\varphi_2}\right]$ has finite rank if and only if $|k|=1$ and $\varphi_{1}\ast_{M}\varphi_{2}=C(r\ast_{M}r^{-1})$ for some
constant $C$.
\end{enumerate}
In each condition, the rank of the commutator is zero.
\end{corollary}
\begin{proof} First we suppose $\left[T_{\varphi_{1}},T_{e^{ik\theta }\varphi_2}\right]$ has finite rank, then Proposition~\ref{prop commutemax} shows that
\begin{align*}
    &2(n+k+1)\widehat{\varphi_1}(2n+2k+2)\widehat{\varphi_2}(2n+k+2)=2(n+1)\widehat{\varphi_1}(2n+2)\widehat{\varphi_2}(2n+k+2)
\end{align*}
holds for any natural number $n\geq \max\left\{0,-k\right\}$. Using the same reasoning as in the proof of \cite[Proposition 6]{LZ1}, one can easily get that either $k=0$ or ${\varphi_{1}}$ is a constant.

Next, suppose $\left[T_{e^{ik\theta}\varphi_{1}}, T_{e^{ik\theta}\varphi_{2}}\right]$ has finite rank. Without loss of generality, we can assume that $k>0$,
for otherwise we could take the adjoints.
Then Proposition~\ref{prop commutemax} shows that
$$\widehat{\varphi_1}(2n+3k+2)\widehat{\varphi_2 }(2n+k+2) =\widehat{\varphi_1}(2n+k+2)\widehat{\varphi_2 }(2n+3k+2)$$
for any $n\in
\mathbb{N}$, which is the same as Equation (3.4) of \cite{DZ3}, and hence $\varphi_{1}=C\varphi_{2}$.

Finally, suppose $\left[T_{e^{ik\theta}\varphi_{1}}, T_{e^{-ik\theta}\varphi_{2}}\right]$ has finite rank. Similarly, we assume $k>0$. Then Proposition~\ref{prop commutemax} shows that
\begin{align*}
    &(2n-2k+2)(2n+2)\widehat{\varphi_1}(2n-k+2)\widehat{\varphi_2 }(2n-k+2)\\
    &=(2n+2)(2n+2k+2)\widehat{\varphi_1}(2n+k+2)\widehat{\varphi_2 }(2n+k+2)
\end{align*}
for any natural number $n\geq k$. We now proceed as in the proof of condition (b) of \cite[Proposition 3.4]{DZ3},
then $(z-k)(z+k)\widehat{\varphi_1}(z)\widehat{\varphi_2 }(z)=C$ for some nonzero constant
$C$.
On the other hand, since ${\varphi_{1}},\;{\varphi_{2}}\in L^{1}([0,1],rdr)$, we get ${\varphi_{1}}\ast_{M}{\varphi_{2}}\in L^{1}([0,1],rdr).$ Therefore
$$\widehat{{\varphi_{1}}\ast_{M}{\varphi_{2}}}(z)=\widehat{\varphi_{1}}(z)\widehat{\varphi_{2}}(z)=\frac{C}{(z-k)(z+k)}$$ must be well defined on $\left\{z:
\textrm{Re}z\geq 2\right\}$, which implies $k=1$ and $\varphi_{1}\ast_{M}\varphi_{2}=C(r\ast_{M}r^{-1})$.

The sufficiency of conditions (a)-(c) is an immediate consequence of \cite[Proposition 3.4]{DZ3}. Furthermore, the commutator is zero in each condition. This completes the proof.
\end{proof}

Now we are ready to state and prove our main theorem in this section, which completely solves the finite rank problem of the commutator $\left[T_{e^{ik_1\theta} r^m},T_{e^{ik_2\theta }\varphi}\right]$ on $L_h^2$.

\begin{theorem}\label{thm commute} Let $k_1,\;k_2\in\mathbb{Z}$, and let $m\in \mathbb{R}$, $m\geq -1$.
Then for a nonzero radial T-function $\varphi(r)$ on $D$, the commutator $\left[T_{e^{ik_1\theta} r^m},T_{e^{ik_2\theta }\varphi}\right]$ has finite rank on $L_h^2$ if and only if one of the following conditions holds:
\begin{enumerate}
\item[(1)]$k_1=m=0$.
\item[(2)]$k_2=0$ and $\varphi=C$.
\item[(3)]$k_1=k_2=0$.
\item[(4)]$k_1=k_2\neq0$ and $\varphi=Cr^m$.
\item[(5)]$k_1k_2=-1$ and $\varphi=C\left(\frac{m+1}{2}r^{-1}-\frac{m-1}{2}r\right)$.
\item[(6)]$k_1k_2<-1$, $|k_2|\geq 2$, $m+|k_1|=0$ and ${\varphi}(r)=Cr^{|k_2|}$.
\item[(7)]$k_1k_2<-1$, $|k_2|=1$ and $\widehat{\varphi}(z)=
C\frac{\Gamma\left(\frac{z-1}{2|k_1|}\right)\Gamma\left(\frac{z+m+|k_1|+1}{2|k_1|}\right)}{\Gamma\left(\frac{z+2|k_1|+1}{2|k_1|}\right)
\Gamma\left(\frac{z+m+|k_1|-1}{2|k_1|}\right)}.$
\item[(8)]$k_1k_2>0$, $k_1\neq k_2$, $|k_2|<m+|k_1|+2$ and $\widehat{\varphi}(z)=
C\frac{\Gamma\left(\frac{z+|k_2|}{2|k_1|}\right)\Gamma\left(\frac{z+m+|k_1|-|k_2|}{2|k_1|}\right)}{\Gamma\left(\frac{z+2|k_1|-|k_2|}{2|k_1|}\right)
\Gamma\left(\frac{z+m+|k_1|+|k_2|}{2|k_1|}\right)}.$
\item[(9)]$k_1k_2>0$, $|k_2|\geq m+|k_1|+2$, $m=(2n+1)|k_1|$ for some $n\in\mathbb{N}$ and
$$\widehat{\varphi}(z)=C\frac{\left(\frac{z-|k_2|}{2|k_1|}+1\right)\left(\frac{z-|k_2|}{2|k_1|}+2\right)\cdots\left(\frac{z-|k_2|}{2|k_1|}+n\right)}
{\frac{z+|k_2|}{2|k_1|}\left(\frac{z+|k_2|}{2|k_1|}+1\right)\cdots\left(\frac{z+|k_2|}{2|k_1|}+n\right)}.$$
\end{enumerate}
In each condition $(1)$-$(5)$, $rank\left(\left[T_{e^{ik_1\theta} r^m},T_{e^{ik_2\theta }\varphi}\right]\right)=0$. In each condition $(6)$-$(9)$,
\begin{align}\label{rmrange}
    Ran\left(\left[T_{e^{ik_1\theta} r^m},T_{e^{ik_2\theta }\varphi}\right]\right)=\emph{span}\left\{r^{|k|}e^{ik\theta}:k\in{{\Lambda_1}}\right\}
\end{align} and
\begin{align*}
    \left[T_{e^{ik_1\theta} r^m},T_{e^{ik_2\theta }\varphi}\right]&=\sum\limits_{k \in {\Lambda_{\frac{1}{2}}} }{C_k \left[ \left(r^{|k|}e^{ik\theta}\right)  \otimes \left(r^{|k-k_1-k_2|}e^{i(k-k_1-k_2)\theta}\right)\right. }\nonumber\\
    &\ \ \ \ \ \ \ \ \ \ \ \ \ \ \ {\left.- \left(r^{|-k+k_1+k_2|}e^{i\left(-k+k_1+k_2\right)\theta}\right)\otimes \left(r^{|-k|}e^{-ik\theta}\right) \right]}
\end{align*}
for some nonzero constant $C_k$.
Furthermore, $rank\left(\left[T_{e^{ik_1\theta} r^m},T_{e^{ik_2\theta }\varphi}\right]\right)=|k_1|+|k_2|-1$ when $k_1+k_2$ is odd, and $rank\left(\left[T_{e^{ik_1\theta} r^m},T_{e^{ik_2\theta }\varphi}\right]\right)=|k_1|+|k_2|-2$ when $k_1+k_2$ is even.
\end{theorem}
\begin{proof} If $k_1k_2=0$ or $|k_1|=|k_2|$, then Corollary~\ref{cor cradial} shows that the commutator $\left[T_{e^{ik_1\theta} r^m},T_{e^{ik_2\theta }\varphi}\right]$ has finite rank if and only if one of the conditions (1)-(5) holds,
and $rank\left(\left[T_{e^{ik_1\theta} r^m},T_{e^{ik_2\theta }\varphi}\right]\right)=0$ in each condition.

Now, we suppose $k_1k_2\neq0$ and $|k_1|\neq |k_2|$. Without loss of generality, we can also assume that $k_1>0$, for otherwise we could take the
adjoints. Then by Proposition~\ref{prop commutemax}, the commutator $\left[T_{e^{ik_1\theta} r^m},T_{e^{ik_2\theta }\varphi}\right]$ has finite rank if and only if
\begin{align}\label{rm}
&2(n+k_2+1)\widehat{r^m}(2n+k_1+2k_2+2)\widehat{\varphi}(2n+k_2+2)\nonumber\\
&=2(n+k_1+1)\widehat{r^m}(2n+k_1+2)\widehat{\varphi}(2n+2k_1+k_2+2)
\end{align}
holds for any natural number $n\geq N_1=\max\left\{0,-k_2\right\}$. Then it follows that
$$
\widehat{\varphi}(2n+2k_1+k_2+2)=\widehat{\varphi
}(2n+k_2+2)\frac{(2n+2k_2+2)(2n+k_1+m+2)}{(2n+2k_1+2)(2n+k_1+2k_2+m+2)},
$$
which is the same as Equation (2.4) of \cite{CR}, so one can get
\begin{equation}\label{Cgk}
\widehat{\varphi}(z)=
C\frac{\Gamma\left(\frac{z+k_2}{2k_1}\right)\Gamma\left(\frac{z+m+k_1-k_2}{2k_1}\right)}{\Gamma\left(\frac{z+2k_1-k_2}{2k_1}\right)
\Gamma\left(\frac{z+m+k_1+k_2}{2k_1}\right)}
\end{equation}
for some nonzero constant $C$.
Thus, according to Lemma~\ref{lem vak}, the commutator $\left[T_{e^{ik_1\theta} r^m},T_{e^{ik_2\theta }\varphi}\right]$ has finite rank if and only if one of the conditions (6)-(9) holds.

Comparing with Proposition~\ref{prop commutemax}, this theorem will be proved if we can show that \eqref{rmrange} holds in each condition (6)-(9).
Moreover, according to the proof of Proposition~\ref{prop commutemax}, we only need to show that $r^{|k|}e^{ik\theta}\in Ran\left(\left[T_{e^{ik_1\theta}r^m},T_{e^{ik_2\theta }\varphi_2}\right]\right)$ holds for any $k\in {\Lambda_{\frac{1}{2}}}$,
or equivalently, $$ \left[T_{e^{ik_1\theta} r^m},T_{e^{ik_2\theta }\varphi}\right]\left(r^{|k-k_1-k_2|}e^{i\left(k-k_1-k_2\right)\theta}\right)
=C_k r^{|k|}e^{ik\theta}$$ holds for some nonzero constant $C_k$. Recall that $k\in {\Lambda_{\frac{1}{2}}}$ implies $\frac{-k_1-k_2}{2}<k-k_1-k_2\leq N_1-1$.
So it suffices to prove that, for any
$l\in\left\{\left[\frac{-k_1-k_2}{2}\right],\cdots,N_1-1\right\},$ where $[\ \cdot\ ]$ denotes the greatest integer function,
$$\left[T_{e^{ik_1\theta} r^m},T_{e^{ik_2\theta }\varphi}\right]\left(r^{|l|}e^{il\theta}\right)=0\Longleftrightarrow l=\frac{-k_1-k_2}{2}.
$$
To show this we need to discuss four cases.

\emph{Case} 1. Suppose $k_1=-m=1$, $k_2\leq -2$. Thus $N_1=-k_2$.
So for any $\left[\frac{-1-k_2}{2}\right]\leq n\leq -1-k_2$, by Lemma~\ref{lem calc} we get
\begin{align*}
    \left[T_{e^{i\theta}r^{-1}},T_{e^{ik_2\theta
}\varphi}\right]({z}^{n})=0&\Longleftrightarrow (-k_2-n+1)\widehat{r^m}(-k_1-2k_2-2n+2)\widehat{\varphi}(-k_2+2)\\
&\ \ \ \ \ \ =(k_1+n+1)\widehat{r^m}(2n+k_1+2)\widehat{\varphi}(-k_2+2)\\
&\Longleftrightarrow \frac{-k_2-n+1}{-2k_2-2n}=\frac{n+2}{2n+2}\Longleftrightarrow n=\frac{-1-k_2}{2},
\end{align*}
as desired.

\emph{Case} 2. Suppose $k_1>1$ and $k_2=-1$. Thus $N_1=1$.
So for any $0\leq n\leq \left[\frac{k_1+k_2}{2}\right]$, by Lemma~\ref{lem calc} we get
\begin{align}\label{zgl3}
    \left[T_{e^{ik_1\theta} r^m},T_{e^{ik_2\theta
}\varphi}\right](\overline{z}^{n})=0
&\Longleftrightarrow (n-k_2+1)\widehat{r^m}(k_1+2)\widehat{\varphi}(2n-k_2+2)\nonumber\\
&\ \ \ \ \ \ =(k_1-n+1)\widehat{\varphi}(2k_1+k_2-2n+2)\widehat{r^m}(k_1+2)\nonumber\\
&\Longleftrightarrow F(n)=F(k_1+k_2-n),
\end{align}
where
\begin{equation}\label{F}
    F(x)=2(x-k_2+1)\widehat{\varphi}(2x-k_2+2).
\end{equation}
Let $a=\frac{k_2}{k_1}$ and $b=\frac{m+k_1}{2k_1}$.
Then by \eqref{Cgk}, we get
\begin{align*}
    F(x)&=2k_1C\frac{2 x-2k_2+2}{2k_1}\frac{\Gamma\left(\frac{2x+2}{2k_1}\right)\Gamma\left(\frac{2x+m+k_1-2k_2+2}{2k_1}\right)}{\Gamma\left(\frac{2x+2k_1-2k_2+2}{2k_1}\right)
\Gamma\left(\frac{2x+m+k_1+2}{2k_1}\right)}\\
    &=2k_1C\frac{\Gamma\left(\frac{x+1}{k_1}\right)\Gamma\left(\frac{x+1}{k_1}+b-a\right)}{\Gamma\left(\frac{x+1}{k_1}-a\right)\Gamma\left(\frac{x+1}{k_1}+b\right)}.
\end{align*}
Note that $a<0$ and $b>0$, so condition (b) of Lemma~\ref{lem G1} implies that $F(x)$ is strictly monotone on $(-1,+\infty)$. Thus \eqref{zgl3} holds if and only if
$n=\frac{k_1+k_2}{2}$, as desired.

\emph{Case} 3. Suppose $k_1>0$, $k_2>0$ and $k_1<k_2$. Thus $N_1=0$.
Observe that both condition (8) and condition (9) imply that $m+k_1\neq 0$.
So for any $k_1\leq n\leq\left[\frac{k_1+k_2}{2}\right]$, by Lemma~\ref{lem calc} we get
\begin{align*}
\left[T_{e^{ik_1\theta} r^m},T_{e^{ik_2\theta
}\varphi}\right](\overline{z}^{n})=0&\Longleftrightarrow (k_2-n+1)\widehat{r^m}(k_1+2k_2-2n+2)\widehat{\varphi}(k_2+2)\\
&\ \ \ \ \ \ =(n-k_1+1)\widehat{r^m}(2n-k_1+2)\widehat{\varphi}(k_2+2)\\
&\Longleftrightarrow \frac{k_2-n+1}{m+k_1+2k_2-2n+2}=\frac{n-k_1+1}{m+2n-k_1+2}\\
&\Longleftrightarrow \frac{m+k_1}{k_2-n+1}=\frac{m+k_1}{n-k_1+1}\Longleftrightarrow \ n=\frac{k_1+k_2}{2}.
\end{align*}
Next, for any $1\leq n<k_1$, by Lemma~\ref{lem calc} we get
\begin{align}\label{Geq}
\left[T_{e^{ik_1\theta} r^m},T_{e^{ik_2\theta
}\varphi}\right](\overline{z}^{n})=0
&\Longleftrightarrow 2(k_2-n+1)\widehat{r^m}(k_1+2k_2-2n+2)\widehat{\varphi}(k_2+2)\nonumber\\
&\ \ \ \ \ \ =2(k_1-n+1)\widehat{r^m}(k_1+2)\widehat{\varphi}(2k_1+k_2-2n+2)\\
&\Longleftrightarrow \frac{2(k_2-n+1)}{m+k_1+2k_2-2n+2}\frac{\Gamma\left(\frac{ 2k_2+2}{2k_1}\right)\Gamma\left(\frac{m+k_1+2}{2k_1}\right)}{\Gamma\left(\frac{2k_1+2}{2k_1}\right)\Gamma\left(\frac{m+k_1+2k_2+2}{2k_1}\right)}\nonumber\\
&\ \ \ \ \ \ =\frac{2(k_1-n+1)}{m+k_1+2}\frac{\Gamma\left(\frac{2k_1+2k_2-2n+2}{2k_1}\right)\Gamma\left(\frac{m+3k_1-2n+2}{2k_1}\right)}
{\Gamma\left(\frac{4k_1-2n+2}{2k_1}\right)\Gamma\left(\frac{m+3k_1+2k_2-2n+2}{2k_1}\right)}.\nonumber
\end{align}
Let $a=\frac{k_2}{k_1}, \ b=\frac{m+k_1}{2k_1} \  \textrm{and}\ x=\frac{n}{k_1}.$ Then the above equation becomes
\begin{align*}
&\frac{\frac{1}{k_1}-x+a}{\frac{1}{k_1}-x+b+a}
\frac{\Gamma\left(\frac{1}{k_1}+a\right)\Gamma\left(\frac{1}{k_1}+b\right)}{\Gamma\left(\frac{1}{k_1}+1\right)\Gamma\left(\frac{1}{k_1}+b+a\right)}
=\frac{\frac{1}{k_1}-x+1}{\frac{1}{k_1}+b}\frac{\Gamma\left(\frac{1}{k_1}-x+a+1\right)\Gamma\left(\frac{1}{k_1}-x+b+1\right)}
{\Gamma\left(\frac{1}{k_1}-x+2\right)\Gamma\left(\frac{1}{k_1}-x+b+a+1\right)}\\
&\Longleftrightarrow \frac{\Gamma\left(\frac{1}{k_1}-x+a\right)\Gamma\left(\frac{1}{k_1}-x+b+1\right)}
{\Gamma\left(\frac{1}{k_1}-x+1\right)\Gamma\left(\frac{1}{k_1}-x+b+a\right)}
=\frac{\Gamma\left(\frac{1}{k_1}+a\right)\Gamma\left(\frac{1}{k_1}+b+1\right)}{\Gamma\left(\frac{1}{k_1}+1\right)\Gamma\left(\frac{1}{k_1}+b+a\right)}\\
&\Longleftrightarrow G\left(x-\frac{1}{k_1}\right)=G\left(-\frac{1}{k_1}\right),
\end{align*}
where the function $G(x)$ is defined by \eqref{G}.
Obviously, this is contradict to condition (b) of Lemma~\ref{lem Gamma} since $a>1$, $b>0$ and $x\in(0,1)$, as desired.

\emph{Case} 4. Suppose $k_1>0$, $k_2>0$ and $k_1>k_2$. Thus $N_1=0$.
So for any $k_2\leq n\leq \left[\frac{k_1+k_2}{2}\right]$, by Lemma~\ref{lem calc} we get
\begin{align}\label{zgl2}
\left[T_{e^{ik_1\theta} r^m},T_{e^{ik_2\theta
}\varphi}\right](\overline{z}^{n})=0&\Longleftrightarrow 2(n-k_2+1)\widehat{r^m}(k_1+2)\widehat{\varphi}(2n-k_2+2)\nonumber\\
&\ \ \ \ \ \ =2(k_1-n+1)\widehat{\varphi}(2k_1+k_2-2n+2)\widehat{r^m}(k_1+2)\nonumber\\
&\Longleftrightarrow F(n)=F(k_1+k_2-n),
\end{align}
where $F(x)$ is also defined by \eqref{F}. Here $a=\frac{k_2}{k_1}>0$ and $b=\frac{m+k_1}{2k_1}>0$.
Then condition (a) of Lemma~\ref{lem G1} implies that $F(x)$ is strictly monotone on $(k_2-1,+\infty)$, and hence \eqref{zgl2} holds if and only if
$n=\frac{k_1+k_2}{2}$, as desired.

Next, for any $1\leq n<k_2$, by Lemma~\ref{lem calc} we get
\begin{align*}
    \left[T_{e^{ik_1\theta} r^m},T_{e^{ik_2\theta
}\varphi}\right](\overline{z}^{n})=0
&\Longleftrightarrow 2(k_2-n+1)\widehat{r^m}(k_1+2k_2-2n+2)\widehat{\varphi}(k_2+2)\\
&\ \ \ \ \ \ =2(k_1-n+1)\widehat{r^m}(k_1+2)\widehat{\varphi}(2k_1+k_2-2n+2),
\end{align*}
which is the same as \eqref{Geq}. Using the same notation, we get that the above equation is also equivalent to
$ G\left(x-\frac{1}{k_1}\right)=G\left(-\frac{1}{k_1}\right)$.
Here $a=\frac{k_2}{k_1}\in(0,1),\ b=\frac{m+k_1}{2k_1}>0\  \textrm{and} \ x=\frac{n}{k_1}\in\left(0,a\right).$
Obviously, this is contradict to condition (a) of Lemma~\ref{lem Gamma}. This completes the proof.
\end{proof}

Below we present the examples of nonzero commutators of finite rank, corresponding to the conditions (6)-(9) of Theorem~\ref{thm commute}.
\begin{example}\label{ex com} Let $\varphi(r)$ be a nonzero radial T-function on $D$, then on $L_h^2$, the following statements hold.
\begin{enumerate}
\item[(1)] The commutator $\left[T_{e^{i\theta} r^{-1}},T_{e^{-3i\theta}\varphi(r)}\right]$ has finite rank if and only if $\varphi(r)=Cr^3.$\\ In this case, $\left[T_{e^{i\theta} r^{-1}},T_{e^{-3i\theta}\varphi(r)}\right]=\frac{1}{2}C\left(1\otimes {z}^2-\overline{z}^2 \otimes 1\right).$
\item[(2)] The commutator $\left[T_{e^{2i\theta} r^{6}},T_{e^{-i\theta}\varphi(r)}\right]$ has finite rank if and only if $\varphi(r)=C\left(\frac{3}{r}-r^3\right).$\\ In this case, $\left[T_{e^{2i\theta} r^{6}},T_{e^{-i\theta}\varphi(r)}\right]=-\frac{19}{30}C\left(1  \otimes \overline{z}-z  \otimes 1\right)$.
\item[(3)] Let $m\in \mathbb{R}$, $m\geq -1$. Then the commutator $\left[T_{e^{i\theta} r^{m}},T_{e^{2i\theta}\varphi(r)}\right]$ has finite rank if and only if $m>-1$ and  $\varphi(r)=C\left[(m+1)r^{m+1}-(m-1)r^{m-1}\right].$\\ In this case,  $\left[T_{e^{i\theta} r^{m}},T_{e^{2i\theta}\varphi(r)}\right]=-\frac{192(m+1)}{(m+5)^2(m+3)^2}C\left({z}\otimes \overline{z}^2-{z}^2\otimes \overline{z}\right)$.
\item[(4)] The commutator $\left[T_{e^{i\theta} r^3},T_{e^{6i\theta}\varphi(r)}\right]$ has finite rank if and only if $\varphi(r)=C\left[6r^{8}-5r^{6}\right].$ In this case, $\left[T_{e^{i\theta} r^3},T_{e^{6i\theta}\varphi(r)}\right]=C\left[-\frac{5}{24}\left({z}\otimes \overline{z}^6-{z}^6\otimes \overline{z}\right)-\frac{27}{196}\left({z}^2\otimes \overline{z}^5-{z}^5\otimes \overline{z}^2\right)\right.$
$\left.-\frac{1}{21}\left({z}^3\otimes \overline{z}^4-{z}^4\otimes \overline{z}^3\right)\right].
$
\end{enumerate}
\end{example}

Next, we show some interesting applications of Theorem~\ref{thm commute}.
\begin{corollary}\label{cor commonial}  Let $k_1,\;k_2\in\mathbb{Z}$ such that $k_1>0$ or $k_1=-1$,
and let $\varphi(r)$ be a nonzero radial T-function on $D$. Then on $L_h^2$, the following statements hold.
\begin{enumerate}
\item[(a)] The commutator $\left[T_{e^{ik_1\theta }r^{k_1}},T_{e^{ik_2\theta }\varphi}\right]$ has finite rank if and only if
$$ k_2>-2\ \ \textrm{and}\ \  \varphi(r)=Cr^{k_2}.$$
\item[(b)] The commutator $\left[T_{e^{-ik_1\theta }r^{k_1}},T_{e^{ik_2\theta }\varphi}\right]$ has finite rank if and only if
$$ k_2<2\ \ \textrm{and}\ \  \varphi(r)=Cr^{-k_2}.$$
\end{enumerate}
\end{corollary}
\begin{proof} If $k_1>0$, then by Theorem~\ref{thm commute}, the commutator $\left[T_{e^{ik_1\theta }r^{k_1}},T_{e^{ik_2\theta }\varphi}\right]$ has finite rank if and only if $k_2$ satisfies one of the conditions (2), (4)-(5) and (7)-(9), and $\widehat{\varphi}(z)=
\frac{C}{z+k_2}.$ In other words, $k_2>-2$ and $\varphi(r)=Cr^{k_2}.$

Similarly, if $k_1=-1$, then $\left[T_{e^{ik_1\theta }r^{k_1}},T_{e^{ik_2\theta }\varphi}\right]$ has finite rank if and only if one of the conditions (2) and (4)-(6) holds, which also implies that $k_2>-2$ and $\varphi(r)=Cr^{k_2}$.

Combining condition (a) with the use of
adjoint operators, one can get condition (b) holds.
\end{proof}
\begin{corollary}\label{cor comr}  Let $k_1,\;k_2\in\mathbb{Z}$ such that $k_1k_2\neq 0$, and let $m_1,\;m_2\in \mathbb{R}$ such that greater than or equal to $-1$.
Then the commutator $\left[T_{e^{ik_1\theta }r^{m_1}},T_{e^{ik_2\theta }r^{m_2}}\right]$ has finite rank on $L_h^2$ if and only if
one of the following conditions holds:
\begin{enumerate}
\item[(1)]$k_1=k_2$ and $m_1=m_2$.
\item[(2)]$k_1=m_1$ and $k_2=m_2$.
\item[(3)]$k_1=-m_1$ and $k_2=-m_2$.
\end{enumerate}
\end{corollary}
\begin{proof} Without loss of generality, we can assume that $k_1>0$. Since $k_2\neq0$ and $$\frac{k_2}{2k_1}+\frac{m_1+k_1-k_2}{2k_1}=\frac{2k_1-k_2}{2k_1}+\frac{m_1+k_1+k_2}{2k_1}-1,$$ one can easily see that $$
\widehat{\varphi}(z)=C\frac{\Gamma\left(\frac{z+k_2}{2k_1}\right)\Gamma\left(\frac{z+m_1+k_1-k_2}{2k_1}\right)}{\Gamma\left(\frac{z+2k_1-k_2}{2k_1}\right)
\Gamma\left(\frac{z+m_1+k_1+k_2}{2k_1}\right)}$$ is a proper rational function with the degree of the denominator one if and only if $$\frac{2k_1-2k_2}{2k_1}=0,\ \ \frac{k_1-m_1}{2k_1}=0\ \ \textrm{or} \ \ \frac{m_1+k_1}{2k_1}=0.$$  Therefore, $\widehat{\varphi}(z)=\widehat{r^{m_2}}(z)=\frac{1}{z+m_2}$ holds if and only if
one of the following conditions holds:
\begin{enumerate}
\item[(1)]$k_1=k_2>0$ and $m_1=m_2$.
\item[(2)]$k_1=m_1>0$ and $k_2=m_2$.
\item[(3)]$k_1=-m_1=1$ and $k_2=-m_2$.
\end{enumerate}
Then by Theorem~\ref{thm commute}, the desired result is obvious.
\end{proof}
\begin{remark}By Theorem~\ref{thm commute}, one can easily get the range, the canonical form and the rank of the finite rank commutator appeared in Corollary~\ref{cor commonial} and Corollary~\ref{cor comr}. For example, for positive natural numbers $k_1$ and $k_2$, we have ${\Lambda_1}=\left\{k\in{\mathbb{Z}}: 1\leq k\leq k_1+k_2-1, k\neq \frac{k_1+k_2}{2}\right\}$, and hence
\begin{equation}\label{comana}
\left[T_{z^{k_1}},T_{z^{k_2}}\right]=\sum\limits_{k \in {\Lambda_1} }{C_k z^{k_1+k_2-k}  \otimes {\overline{z}}^{k} },
\end{equation}
which is the same as \cite[Proposition 2.3]{CKL}. We would like to point out that in \cite{CKL} the authors first got \eqref{comana} by the direct calculation, taking on the reproducing kernel, then they showed that the commutator $\left[T_{z^{k_1}},T_{z^{k_2}}\right]$ has finite rank. So the method used in this paper is quite different from that in \cite{CKL}.
\end{remark}

\section{Ranks of generalized semicommutators on $L_h^2$}

In this section,
we will discuss the finite rank generalized semicommutators of quasihomogeneous Toeplitz
operators on $L_h^2$. Obviously, if $T_{e^{ik_1\theta}\varphi_{1}}T_{e^{ik_2\theta }\varphi_2}-T_{f}$ has finite rank, then $f$ must be quasihomogeneous of degree $k_1+k_2$. Therefore, it suffices to consider the generalized semicommutator of the form $T_{e^{ik_1\theta}\varphi_{1}}T_{e^{ik_2\theta }\varphi_2}-T_{e^{i(k_1+k_2)\theta }\psi}$ in this section.
Throughout the rest part of this paper, we also define $$N_2=\max\left\{0,-k_2,-k_1-k_2\right\}\  \textrm{and} \ N_3=\max\left\{0,k_2,k_1+k_2\right\}.$$
Then we write $${\Lambda_2}=\left\{k\in{\mathbb{Z}}: -N_3+k_1+k_2+1\leq k\leq N_2+k_1+k_2-1\right\}.$$
First we show the general result of the finite rank generalized semicommutators of two quasihomogeneous Toeplitz
operators.

\begin{proposition}\label{prop pruductmax}  Let $k_1,\;k_2\in\mathbb{Z}$, and let ${\varphi_{1}},\;{\varphi_{2}}$ and $\psi$ be radial T-functions. Then the generalized semicommutator $T_{e^{ik_1\theta}\varphi_{1}}T_{e^{ik_2\theta }\varphi_2}-T_{e^{i(k_1+k_2)\theta }\psi}$ has finite rank on $L_h^2$ if and only if
\begin{equation}\label{pmax1}
2(n+k_2+1)\widehat{\varphi_1}(2n+k_1+2k_2+2)\widehat{\varphi_2}(2n+k_2+2)=\widehat{\psi }(2n+k_1+k_2+2)
\end{equation}
holds for any natural number $n\geq N_2$, and
\begin{equation}\label{pmax2}
2(n-k_2+1)\widehat{\varphi_1 }(2n-k_1-2k_2+2)\widehat{\varphi_2}(2n-k_2+2)=\widehat{\psi }(2n-k_1-k_2+2)
\end{equation}
holds for any natural number $n\geq N_3$.
In this case,
\begin{align*}
    &Ran\left(T_{e^{ik_1\theta}\varphi_{1}}T_{e^{ik_2\theta }\varphi_2}-T_{e^{i(k_1+k_2)\theta }\psi}\right)=
\emph{span}\left\{r^{|k_{(1)}|}e^{ik_{(1)}\theta}, r^{|k_{(2)}|}e^{ik_{(2)}\theta},\cdots, r^{|k_{(N)}|}e^{ik_{(N)}\theta}\right\}
\end{align*}
for some $k_{(1)},\cdots, k_{(N)}\in \Lambda_{2}$,
and \begin{align}\label{sczhengjiaohe}
&T_{e^{ik_1\theta}\varphi_{1}}T_{e^{ik_2\theta }\varphi_2}-T_{e^{i(k_1+k_2)\theta }\psi}=\sum\limits_{j = 1}^N {C_{k_{(j)}} \left(r^{|k_{(j)}|}e^{ik_{(j)}\theta}\right)  \otimes \left(r^{|k_{(j)}-k_1-k_2|}e^{i\left(k_{(j)}-k_1-k_2\right)\theta}\right) }
\end{align}
for some nonzero constant $C_{k_{(j)}}$.
Furthermore,
$$rank\left(T_{e^{ik_1\theta}\varphi_{1}}T_{e^{ik_2\theta }\varphi_2}-T_{e^{i(k_1+k_2)\theta }\psi}\right)\leq \max\left\{0,|k_1|-1,|k_2|-1,|k_1+k_2|-1\right\}.
$$
\end{proposition}
\begin{proof}
By Lemma~\ref{lem calc} we get
\begin{align*}
    &\left(T_{e^{ik_1\theta}\varphi_{1}}T_{e^{ik_2\theta }\varphi_2}-T_{e^{i(k_1+k_2)}\psi}\right)(z^n)\\
    &=\left[2(n+k_2+1)\widehat{\varphi_1}(2n+k_1+2k_2+2)\widehat{\varphi_2}(2n+k_2+2)
    -\widehat{\psi }(2n+k_1+k_2+2)\right]\\
    &\ \ \ \ \ \times 2(n+k_1+k_2+1)z^{n+k_1+k_2}
\end{align*} holds for any natural number $n\geq N_2$, and
\begin{align*}
    &\left(T_{e^{ik_1\theta}\varphi_{1}}T_{e^{ik_2\theta }\varphi_2}-T_{e^{i(k_1+k_2)}\psi}\right)(\overline{z}^n)\\
    &=\left[2(n-k_2+1)\widehat{\varphi_1 }(2n-k_1-2k_2+2)\widehat{\varphi_2}(2n-k_2+2)-\widehat{\psi }(2n-k_1-k_2+2)\right] \\
    &\ \ \ \ \ \ \times 2(n-k_1-k_2+1)\overline{z}^{n-k_1-k_2}
\end{align*} holds for any natural number $n\geq N_3$. As in the proof of Proposition~\ref{prop commutemax}, one can easily get that $T_{e^{ik_1\theta} r^m}T_{e^{ik_2\theta }\varphi}-T_{e^{i(k_1+k_2)}\psi}$ has
finite rank if and only if \eqref{pmax1} holds for $ n\geq N_2$ and \eqref{pmax2} holds for $ n\geq N_3$.
Therefore,
\begin{align*}
    &Ran\left(T_{e^{ik_1\theta}\varphi_{1}}T_{e^{ik_2\theta }\varphi_2}-T_{e^{i(k_1+k_2)\theta }\psi}\right)\\
    &=\emph{span}\left\{\left(T_{e^{ik_1\theta}\varphi_{1}}T_{e^{ik_2\theta }\varphi_2}-T_{e^{i(k_1+k_2)\theta }\psi}\right)\left(r^{|l|}e^{il\theta}\right):l\in{\mathbb{Z}}, -N_3+1\leq l\leq N_2-1\right\}\\
&\subseteq
\emph{span}\left\{r^{|k|}e^{ik\theta}:k\in{\Lambda_{2}}\right\}.
\end{align*}
As a direct consequence of Lemma~\ref{lem comrank}, then we get \eqref{sczhengjiaohe} holds.
Moreover, by simple calculations, one can get $$N_2+N_3=\max\left\{|k_1|,|k_2|,|k_1+k_2|\right\}, $$ as desired. This completes the proof.
\end{proof}

\begin{remark}\label{cor csrelation}
Combining \eqref{S} with the use of adjoint operators, it follows from \eqref{sczhengjiaohe} that
\begin{align*}
&T_{e^{ik_2\theta }\varphi_2}T_{e^{ik_1\theta}\varphi_{1}}-T_{e^{i(k_1+k_2)\theta }\psi}=\sum\limits_{j = 1}^N {C_{k_{(j)}} \left(r^{|-k_{(j)}+k_1+k_2|}e^{i\left(-k_{(j)}+k_1+k_2\right)\theta} \right) \otimes\left(r^{|-k_{(j)}|}e^{-ik_{(j)}\theta}\right) };\\
&T_{e^{-ik_1\theta} \overline{\varphi_{1}}}T_{{e^{-ik_2\theta }\overline{\varphi_2}}}-T_{e^{-i(k_1+k_2)\theta} \overline{\psi}}=\sum\limits_{j = 1}^N {\overline{C_{k_{(j)}}} \left(r^{|-k_{(j)}|}e^{-ik_{(j)}\theta}\right)  \otimes \left(r^{|-k_{(j)}+k_1+k_2|}e^{i\left(-k_{(j)}+k_1+k_2\right)\theta} \right)  };\\
&T_{{e^{-ik_2\theta }\overline{\varphi_2}}}T_{e^{-ik_1\theta} \overline{\varphi_{1}}}-T_{e^{-i(k_1+k_2)\theta} \overline{\psi}}=\sum\limits_{j = 1}^N {\overline{C_{k_{(j)}}} \left(r^{|k_{(j)}-k_1-k_2|}e^{i\left(k_{(j)}-k_1-k_2\right)\theta}\right) \otimes \left(r^{|k_{(j)}|}e^{ik_{(j)}\theta}\right)}.
\end{align*}
Furthermore, one can get
\begin{align*}
\left[T_{e^{ik_1\theta}\varphi_{1}},T_{e^{ik_2\theta }\varphi_2}\right]
&=\left(T_{e^{ik_1\theta}\varphi_{1}}T_{e^{ik_2\theta }\varphi_2}-T_{e^{i(k_1+k_2)\theta }\psi}\right)-\left(T_{e^{ik_2\theta }\varphi_2}T_{e^{ik_1\theta}\varphi_{1}}-T_{e^{i(k_1+k_2)\theta }\psi}\right)\\
&=\sum\limits_{j = 1}^N {C_{k_{(j)}}\left[ \left(r^{|k_{(j)}|}e^{ik_{(j)}\theta}\right)  \otimes \left(r^{|k_{(j)}-k_1-k_2|}e^{i\left(k_{(j)}-k_1-k_2\right)\theta}\right) \right.}\\ &\ \ \ \ \ \ \ \ \ \ \ \ \ \ \  \ \left.-\left(r^{|-k_{(j)}+k_1+k_2|}e^{i\left(-k_{(j)}+k_1+k_2\right)\theta} \right) \otimes\left(r^{|-k_{(j)}|}e^{-ik_{(j)}\theta}\right)\right],
\end{align*}
which represents in the same form as \eqref{pair}. But the above form is not necessarily the canonical form, as here $k_{(i)}=-k_{(j)}+k_1+k_2$
may hold for some $k_{(i)},\ k_{(j)}\in \Lambda_{2}$. Thus, the rank of $\left[T_{e^{ik_1\theta} \varphi_{1}},T_{e^{ik_2\theta }\varphi_2}\right]$ should be less or equal to $2\ rank\left(T_{e^{ik_1\theta}\varphi_{1}}T_{e^{ik_2\theta }\varphi_2}-T_{e^{i(k_1+k_2)\theta }\psi}\right),$ which corresponds to Corollary~\ref{cor relation}. In fact, Example~\ref{ex com} and Example~\ref{ex semic} also illustrate this point.
\end{remark}

The following corollary gives a complete description of the finite rank generalized semicommutators of
quasihomogeneous Toeplitz operators with some special
degrees.

\begin{corollary}\label{cor pradial}  Let $k\in\mathbb{Z}$, and let ${\varphi_{1}},\;{\varphi_{2}}$ and $\psi$ be radial T-functions. Then on $L_h^2$, the following statements hold.

\begin{enumerate}
\item[(a)] If $\varphi_{1}$ is nonconstant and $\varphi_{2}\neq0$, then the generalized semicommutator $T_{\varphi_{1}}T_{e^{ik\theta }\varphi_2}-T_{e^{ik\theta }\psi}$ has finite rank if and only if $k=0$ and $\psi$ is a solution of the equation
$$\mathbb{I}\ast_{M}\psi=\varphi_{1}\ast_{M}\varphi_{2}.$$
\item[(b)] If $k\neq0$, then the generalized semicommutator $T_{e^{ik\theta }\varphi_{1}}T_{e^{-ik\theta }\varphi_2}-T_{\psi}$ has finite rank if and only if $|k|=1$, $\varphi_{1}\ast_{M}\varphi_{2}=C(r\ast_{M}r^{-1})$ and $\psi=C$ for some
constant $C$.
\end{enumerate}
In each case, the rank of the generalized semicommutator is zero.
\end{corollary}
\begin{proof} First assume $T_{\varphi_{1}}T_{e^{ik\theta }\varphi_2}-T_{e^{ik\theta }\psi}$ has finite rank, then
combining Corollary~\ref{cor relation} with Corollary~\ref{cor cradial} we get $k=0$. Thus Proposition~\ref{prop pruductmax} gives that
$rank\left(T_{\varphi_{1}}T_{e^{ik\theta }\varphi_2}-T_{e^{ik\theta }\psi}\right)\leq0,$
and so $T_{\varphi_{1}}T_{e^{ik\theta }\varphi_2}=T_{e^{ik\theta }\psi}.$ Then by \cite[Corollary 3.1]{DZ4}, the above equation holds if and only if $\psi$ is a solution of the equation
$\mathbb{I}\ast_{M}\psi=\varphi_{1}\ast_{M}\varphi_{2},$ and hence condition (a) holds.

The proof of (b) is similar. This completes the proof.
\end{proof}

Next, we give our main result in this section, which completely characterize when the generalized semicommutator $T_{e^{ik_1\theta} r^m}T_{e^{ik_2\theta }\varphi}-T_{e^{i(k_1+k_2)\theta }\psi}$ to be finite rank.

\begin{theorem}\label{thm semicom}  Let $k_1,\;k_2\in\mathbb{Z}$, and let $m\in \mathbb{R}$, $m\geq -1$.
Then for nonzero radial T-functions $\varphi$ and $\psi$ on $D$, the generalized semicommutator $T_{e^{ik_1\theta} r^m}T_{e^{ik_2\theta }\varphi}-T_{e^{i(k_1+k_2)\theta }\psi}$ has finite rank on $L_h^2$ if and only if one of the following conditions holds:
\begin{enumerate}
\item[(1)]$k_1=m=0$ and $\psi=\varphi$.
\item[(2)]$k_2=0$, $\varphi=C$ and $\psi=Cr^m$ for some constant $C$.
\item[(3)]$k_1=k_2=0$ and $\psi(r)=\varphi(r)-mr^m\int_r^1{\frac{\varphi(t)}{t^{m+1}}dt}$.
\item[(4)]$k_1k_2=-1$, $\varphi=C\left(\frac{m+1}{2}r^{-1}-\frac{m-1}{2}r\right)$ and $\psi=C$ for some constant $C$.
\item[(5)]$k_1k_2<-1$, $|k_2|\geq 2$, $m+|k_1|=0$, ${\varphi}(r)=Cr^{|k_2|}$ and $\psi=Cr^{|k_2|-1}$.
\item[(6)]$k_1k_2<-1$, $|k_2|=1$, $\widehat{\varphi}(z)=
C\frac{\Gamma\left(\frac{z-1}{2|k_1|}\right)\Gamma\left(\frac{z+m+|k_1|+1}{2|k_1|}\right)}{\Gamma\left(\frac{z+2|k_1|+1}{2|k_1|}\right)
\Gamma\left(\frac{z+m+|k_1|-1}{2|k_1|}\right)}$
and $$\widehat{\psi}(z)=
C\frac{\Gamma\left(\frac{z+|k_1|-1}{2|k_1|}\right)\Gamma\left(\frac{z+m+1}{2|k_1|}\right)}{\Gamma\left(\frac{z+|k_1|+1}{2|k_1|}\right)
\Gamma\left(\frac{z+m+2|k_1|-1}{2|k_1|}\right)}.$$
\item[(7)]$k_1k_2>0$, $|k_2|<m+2$, $\widehat{\varphi}(z)=
C\frac{\Gamma\left(\frac{z+|k_2|}{2|k_1|}\right)\Gamma\left(\frac{z+m+|k_1|-|k_2|}{2|k_1|}\right)}{\Gamma\left(\frac{z+2|k_1|-|k_2|}{2|k_1|}\right)
\Gamma\left(\frac{z+m+|k_1|+|k_2|}{2|k_1|}\right)}$ and $$\widehat{\psi}(z)=C\frac{\Gamma\left(\frac{z+|k_1|+|k_2|}{2|k_1|}\right)\Gamma\left(\frac{z+m-|k_2|}{2|k_1|}\right)}{\Gamma\left(\frac{z+|k_1|-|k_2|}{2|k_1|}\right)
\Gamma\left(\frac{z+m+2|k_1|+|k_2|}{2|k_1|}\right)}.$$
\item[(8)]$k_1k_2>0$, $|k_2|\geq m+2$, $m=(2n+1)|k_1|$ for some $n\in\mathbb{N}$, $$\widehat{\varphi}(z)=C\frac{\left(\frac{z-|k_2|}{2|k_1|}+1\right)\left(\frac{z-|k_2|}{2|k_1|}+2\right)\cdots\left(\frac{z-|k_2|}{2|k_1|}+n\right)}
{\frac{z+|k_2|}{2|k_1|}\left(\frac{z+|k_2|}{2|k_1|}+1\right)\cdots\left(\frac{z+|k_2|}{2|k_1|}+n\right)}$$ and $$\widehat{\psi}(z)=C\frac{\frac{z+|k_1|-|k_2|}{2|k_1|}\left(\frac{z+|k_1|-|k_2|}{2|k_1|}+1\right)\cdots\left(\frac{z+|k_1|-|k_2|}{2|k_1|}+n-1\right)}
{\frac{z+|k_1|+|k_2|}{2|k_1|} \left(\frac{z+|k_1|+|k_2|}{2|k_1|}+1\right)\cdots\left(\frac{z+|k_1|+|k_2|}{2|k_1|}+n\right) }.$$
\end{enumerate}
In each condition $(1)$-$(4)$, $rank\left(T_{e^{ik_1\theta} r^m}T_{e^{ik_2\theta }\varphi}-T_{e^{i(k_1+k_2)\theta }\psi}\right)=0$. In each condition $(5)$-$(8)$,
$$Ran\left(T_{e^{ik_1\theta} r^m}T_{e^{ik_2\theta }\varphi}-T_{e^{i(k_1+k_2)\theta }\psi}\right)
    =\emph{span}\left\{r^{|k|}e^{ik\theta}: k \in {\Lambda_2} \right\}$$ and
$$T_{e^{ik_1\theta} r^m}T_{e^{ik_2\theta }\varphi}-T_{e^{i(k_1+k_2)\theta }\psi}=\sum\limits_{k \in {\Lambda_2} }{C_k \left(r^{|k|}e^{ik\theta}\right)  \otimes \left(r^{|k-k_1-k_2|}e^{i(k-k_1-k_2)\theta}\right) }
$$
for some nonzero constant $C_k$.
Therefore,
$$rank\left(T_{e^{ik_1\theta} r^m}T_{e^{ik_2\theta }\varphi}-T_{e^{i(k_1+k_2)\theta }\psi}\right)=\max\left\{|k_1|-1,|k_2|-1,|k_1+k_2|-1\right\}.$$
\end{theorem}
\begin{proof}First we suppose $k_1 k_2=0$, or $k_1k_2=-1$. Then Corollary~\ref{cor pradial} shows that $T_{e^{ik_1\theta} r^m}T_{e^{ik_2\theta }\varphi}-T_{e^{i(k_1+k_2)\theta }\psi}$ has finite rank if and only if one of the conditions (1)-(4) holds. Moreover, in each condition, \cite[Corollary 3.1]{DZ4} shows that
$T_{e^{ik_1\theta} r^m}T_{e^{ik_2\theta }\varphi}=T_{e^{i(k_1+k_2)\theta }\psi}$, as desired.

Next, we suppose $k_1k_2\neq0$ and $k_1k_2\neq-1$. Combining Theorem~\ref{thm eqvi} with the use of adjoint operators, we can further assume $k_1>0$. Then by Proposition~\ref{prop pruductmax}, $T_{e^{ik_1\theta} r^m}T_{e^{ik_2\theta }\varphi}-T_{e^{i(k_1+k_2)\theta }\psi}$ has finite rank if and only if
\begin{equation}\label{prmax1}
2(n+k_2+1)\widehat{r^m}(2n+k_1+2k_2+2)\widehat{\varphi}(2n+k_2+2)=\widehat{\psi }(2n+k_1+k_2+2)
\end{equation}
holds for all $n\geq N_2=\max\left\{0,-k_2,\right\}$, and
\begin{equation}\label{prmax2}
2(n-k_2+1)\widehat{r^m}(2n-k_1-2k_2+2)\widehat{\varphi}(2n-k_2+2)=\widehat{\psi }(2n-k_1-k_2+2)
\end{equation} holds for all $n\geq N_3=\max\left\{0,k_1+k_2\right\}$.

Now, we assume $T_{e^{ik_1\theta} r^m}T_{e^{ik_2\theta }\varphi}-T_{e^{i(k_1+k_2)\theta }\psi}$ has finite rank. Then by Corollary~\ref{cor relation},  the commutator $\left[T_{e^{ik_1\theta} r^m},T_{e^{ik_2\theta }\varphi}\right]$ has finite rank, and hence one of the conditions (4) and (6)-(9) of Theorem~\ref{thm commute} holds.
On the other hand, \eqref{prmax1} implies
$$\widehat{\psi }(z)=(z-k_1+k_2)\widehat{r^m}(z+k_2)\widehat{\varphi}(z-k_1).
$$
Thus, by \eqref{Cgk} we get
\begin{align*}
   \widehat{\psi }(z)&=C\frac{z-k_1+k_2}{z+m+k_2}\frac{\Gamma\left(\frac{z+k_2-k_1}{2k_1}\right)\Gamma\left(\frac{z+m-k_2}{2k_1}\right)}{\Gamma\left(\frac{z+k_1-k_2}{2k_1}\right)
\Gamma\left(\frac{z+m+k_2}{2k_1}\right)}=C\frac{\Gamma\left(\frac{z+k_1+k_2}{2k_1}\right)\Gamma\left(\frac{z+m-k_2}{2k_1}\right)}{\Gamma\left(\frac{z+k_1-k_2}{2k_1}\right)
\Gamma\left(\frac{z+m+2k_1+k_2}{2k_1}\right)}.
\end{align*}
As in the proof of Lemma~\ref{lem vak}, it is easy to verify that
$\psi(r)$ is a T-function if and only if one of the following conditions holds:
\begin{itemize}
  \item $k_2\leq-k_1-2$ and $m+k_1=0$.
  \item $-k_1-2<k_2<m+2$.
  \item $k_2\geq m+2$ and $m=(2n+1)k_1$ for some $n\in\mathbb{N}$.
\end{itemize}
Then combining with Theorem~\ref{thm commute}, it is easy to show that one of the conditions (5)-(8) holds.

Conversely, if one of the conditions (5)-(8) holds, then \eqref{rm} and \eqref{prmax1} hold for all $n\geq N_1=N_2=\max\left\{0,-k_2,\right\}$. From \eqref{rm} and \eqref{prmax1} we can deduce
$$2(n+k_1+1)\widehat{r^m}(2n+k_1+2)\widehat{\varphi}(2n+2k_1+k_2+2)=\widehat{\psi }(2n+k_1+k_2+2)$$ holds for all $n\geq \max\left\{0,-k_2,\right\}$.
Replacing $n$ by $n-k_1-k_2$, then the above equation implies that \eqref{prmax2} holds for all $n\geq \max\left\{k_1,k_1+k_2,\right\}$.  As in the proof of Proposition~\ref{prop commutemax}, we get \eqref{prmax2} holds for all $n\geq N_3=\max\left\{0,k_1+k_2\right\}$, and hence $T_{e^{ik_1\theta} r^m}T_{e^{ik_2\theta }\varphi}-T_{e^{i(k_1+k_2)\theta }\psi}$ has finite rank.

According to Proposition~\ref{prop pruductmax}, to finish the proof we only need to show that, for any $l\in \left\{-N_3+1,\cdots,N_2-1\right\},$
$\left(T_{e^{ik_1\theta} r^m}T_{e^{ik_2\theta }\varphi}-T_{e^{i(k_1+k_2)}\psi}\right)\left(r^{|l|}e^{il\theta}\right)=C_lr^{|l+k_1+k_2|}e^{i\left(l+k_1+k_2\right)\theta}$ must hold for some nonzero constant $C_l$. To show this we need to discuss three cases.

\emph{Case} 1. Suppose $k_1=-m=1$, $k_2\leq -2$.
Thus $N_2=-k_2$ and $N_3=0$.
So for any $0< n\leq-k_2-1$, by Lemma~\ref{lem calc} we get
\begin{align}\label{prmg1}
    &\left(T_{e^{ik_1\theta} r^m}T_{e^{ik_2\theta}\varphi}-T_{{e^{i(k_1+ k_2)\theta}\psi}}\right)({z}^{n})=0\nonumber\\
&\Longleftrightarrow 2(-k_2-n+1)\widehat{r^m}(-k_1-2k_2-2n+2)\widehat{\varphi}(-k_2+2)=\widehat{\psi}(-k_1-k_2+2).
\end{align}
Denote
\begin{align*}
H(x)&=2(-k_2-x+1)\widehat{r^m}(-k_1-2k_2-2x+2)\\
&=\frac{2(-k_2-x+1)}{m-k_1-2k_2-2x+2}=\frac{1}{1-\frac{1}{-k_2-x+1}},
\end{align*} which implies that $H(x)$ is strictly monotone increasing on $[0,-k_2)$.
Hence $$H(0)<H(1)<\cdots<H(-k_2-1).$$
On the other hand, by taking $n=0$ in \eqref{prmax2}, we obtain $$H(0)\widehat{\varphi }(-k_2+2)=\widehat{\psi }(-k_1-k_2+2),$$
Since $\widehat{\varphi }(-k_2+2)\neq 0$, it follows that \eqref{prmg1} would not hold, as desired.

\emph{Case} 2. Suppose $k_1>0$ and $k_2=-1$.
Thus $N_2=1$ and
$N_3=k_1+k_2$. So for any $0\leq n< k_1+k_2$, by Lemma~\ref{lem calc} we get
\begin{align}\label{prmg2}
    &\left(T_{e^{ik_1\theta} r^m}T_{e^{ik_2\theta}\varphi}-T_{{e^{i(k_1+ k_2)\theta}\psi}}\right)(\overline{z}^{n})=0\nonumber\\
&\Longleftrightarrow  2(n-k_2+1)\widehat{r^m }(k_1+2)\widehat{\varphi}(2n-k_2+2)=\widehat{\psi }(k_1+k_2+2)\nonumber\\
&\Longleftrightarrow  F(n)\widehat{r^m }(k_1+2)=\widehat{\psi }(k_1+k_2+2),
\end{align}
where $F(x)$ is defined by \eqref{F}.
Since $a=\frac{k_2}{k_1}<0$ and $b=\frac{m+k_1}{2k_1}>0,$
it follows that $F(x)$ is strictly monotone decreasing on $(-1,+\infty)$.
Thus, $$F(0)>F(1)>\cdots>F(k_1+k_2).$$
On the other hand, by taking $n=k_1+k_2$ in \eqref{prmax2}, we obtain $$F(k_1+k_2)\widehat{r^m}(k_1+2)=\widehat{\psi }(k_1+k_2+2),$$
and hence \eqref{prmg2} would not hold.

\emph{Case} 3. Suppose $k_1k_2>0$. Moreover, both condition (7) and condition (8) imply that $m+k_1\neq 0$.
Thus $N_2=0$ and $N_3=k_1+k_2$.

First, for any $1\leq n\leq k_2$, by Lemma~\ref{lem calc} we get
\begin{align}\label{prmg3}
    &\left(T_{e^{ik_1\theta} r^m}T_{e^{ik_2\theta}\varphi}-T_{{e^{i(k_1+ k_2)\theta}\psi}}\right)(\overline{z}^{n})=0\nonumber\\
&\Longleftrightarrow  2(k_2 -n+1)\widehat{r^m}(k_1+2k_2-2n+2)\widehat{\varphi }(k_2+2)=\widehat{\psi }(k_1+k_2+2).
\end{align}
Denote
\begin{align*}
P(x)&=2(k_2 -x+1)\widehat{r^m}(k_1+2k_2-2x+2)\\
&=\frac{2(k_2 -x+1)}{m+k_1+2k_2-2x+2}=\frac{1}{\frac{m+k_1}{2(k_2 -x+1)}+1},
\end{align*} which implies that $P(x)$ is strictly monotone increasing on $[0,+\infty)$.
Hence $$P(0)<P(1)<\cdots<P(k_2).$$
On the other hand,  by taking $n=0$ in \eqref{prmax1}, we obtain $$P(0)\widehat{\varphi }(k_2+2)=\widehat{\psi }(k_1+k_2+2).$$
Since $\widehat{\varphi }(k_2+2)\neq 0$, it follows that \eqref{prmg3} would not hold.

Next, for any $k_2\leq n<k_1+k_2$, by Lemma~\ref{lem calc} we get
\begin{align}\label{prmg4}
    &\left(T_{e^{ik_1\theta} r^m}T_{e^{ik_2\theta}\varphi}-T_{{e^{i(k_1+ k_2)\theta}\psi}}\right)(\overline{z}^{n})=0\nonumber\\
&\Longleftrightarrow  2(n-k_2+1)\widehat{r^m }(k_1+2)\widehat{\varphi}(2n-k_2+2)=\widehat{\psi }(k_1+k_2+2)\nonumber\\
&\Longleftrightarrow F(n)\widehat{r^m }(k_1+2)=\widehat{\psi }(k_1+k_2+2).
\end{align}
Since $a=\frac{k_2}{k_1}>0$ and $b=\frac{m+k_1}{2k_1}>0,$
it follows that $F(x)$ is strictly monotone increasing on $(k_2-1,+\infty)$.
Thus, $$F(k_2)<F(k_2+1)<\cdots<F(k_1+k_2).$$
On the other hand, by taking $n=k_1+k_2$ in \eqref{prmax2}, we obtain $$F(k_1+k_2)\widehat{r^m}(k_1+2)=\widehat{\psi }(k_1+k_2+2),$$
and hence \eqref{prmg4} would not hold, as desired. This completes the proof.
\end{proof}

Below we present the examples of nonzero generalized semicommutators of finite rank, corresponding to the conditions (5)-(8) of Theorem~\ref{thm semicom}.

\begin{example}\label{ex semic}  Let $\varphi(r)$
be a nonzero radial T-function on $D$, then on $L_h^2$, the following statements hold.
\begin{enumerate}
\item[(1)] $T_{e^{i\theta} r^{-1}}T_{e^{-3i\theta}\varphi(r)}-T_{{e^{-2i\theta}\psi}}$ has finite rank if and only if
$$\varphi(r)=Cr^3\ \ \textrm{and}\ \  \psi(r)=Cr^2.$$ In this case, $T_{e^{i\theta} r^{-1}}T_{e^{-3i\theta}\varphi(r)}-T_{{e^{-2i\theta}\psi}}=C\left[\frac{1}{2}\left(1 \otimes z^2\right)+\frac{1}{6}\left(\overline{z}\otimes z\right)\right].$
\item[(2)] $T_{e^{2i\theta} r^{6}}T_{e^{-i\theta}\varphi(r)}-T_{{e^{i\theta}\psi}}$ has finite rank if and only if $$\varphi(r)=C\left(\frac{3}{r}-r^3\right)\ \ \textrm{and}\ \  \psi(r)=C(r+r^5).$$ In this case,
    $T_{e^{2i\theta} r^{6}}T_{e^{-i\theta}\varphi(r)}-T_{{e^{i\theta}\psi}}=\frac{19C}{30}\left(z\otimes 1\right).$
\item[(3)]Let $m\in \mathbb{R}$, $m\geq -1$. Then $T_{e^{i\theta} r^{m}}T_{e^{2i\theta}\varphi(r)}-T_{{e^{3i\theta}\psi}}$ has finite rank if and only if $m>0$, $$\varphi(r)=C\left[(m+1)r^{m+1}-(m-1)r^{m-1}\right] \ \textrm{and}$$
$$\psi(r)=\frac{C}{4}\left[(m-3)(m-1)r^{m-2}-(m-1)(m+1)r^{m}+(m+1)(m+3)r^{m+2}\right].$$
In this case,
$T_{e^{i\theta} r^{m}}T_{e^{2i\theta}\varphi(r)}-T_{{e^{3i\theta}\psi}}=-\frac{192(m+1)C}{(m+7)(m+5)(m+3)}\left[\frac{2}{m+3}\left(z\otimes \overline{z}^2\right)+\frac{1}{m+5}\left(z^2\otimes \overline{z}\right)\right].$
\item[(4)]$T_{e^{i\theta} r^{3}}T_{e^{6i\theta}\varphi(r)}-T_{{e^{7i\theta}\psi}}$ has finite rank if and only if $$\varphi(r)=C\left[6r^{8}-5r^{6}\right]\ \ \textrm{and} \ \     \psi(r)=C\left[7r^{9}-6r^{7}\right].$$ In this case,
\begin{align*}
T_{e^{i\theta} r^{3}}T_{e^{6i\theta}\varphi(r)}-T_{{e^{7i\theta}\psi}}=&C\left[-\frac{2}{9}\left({z}\otimes \overline{z}^6\right)-\frac{1}{72}\left({z}^6\otimes \overline{z}\right)-\frac{5}{28}\left({z}^2\otimes \overline{z}^5\right)\right.\\
&\ \ \ \ \left.-\frac{2}{49}\left({z}^5\otimes \overline{z}^2\right)-\frac{8}{63}\left({z}^3\otimes \overline{z}^4\right)-\frac{5}{63}\left({z}^4\otimes \overline{z}^3\right)\right].
\end{align*}
\end{enumerate}
\end{example}

Finally, we show some interesting applications of Theorem~\ref{thm semicom}.

\begin{corollary}\label{cor semianaly}  Let $k_1,\;k_2\in\mathbb{Z}$ such that $k_1>0$ or $k_1=-1$,
and let $\varphi(r),\;\psi(r)$ be two nonzero radial T-functions on $D$. Then on $L_h^2$, the following statements hold.
\begin{enumerate}
\item[(a)] $T_{e^{ik_1\theta}r^{k_1}} T_{e^{ik_2\theta}\varphi(r)}-T_{{e^{i(k_1+k_2)\theta}\psi}}$ has finite rank if and only if $$k_2>\max\{-2,-k_1-2\},\ \ \varphi(r)=Cr^{k_2}\ \ \textrm{and}\ \ \psi(r)=Cr^{k_1+k_2}.$$
\item[(b)] $T_{e^{-ik_1\theta}r^{k_1}}T_{e^{ik_2\theta}\varphi(r)}-T_{{e^{i(-k_1+k_2)\theta}\psi}}$ has finite rank if and only if $$k_2<\min\{2,k_1+2\},\ \ \varphi(r)=Cr^{-k_2}\ \ \textrm{and}\ \ \psi(r)=Cr^{k_1-k_2}.$$
\end{enumerate}
\end{corollary}
\begin{proof}
If $k_1>0$, then by Theorem~\ref{thm semicom}, $T_{e^{ik_1\theta}r^{k_1}} T_{e^{ik_2\theta}\varphi(r)}-T_{{e^{i(k_1+k_2)\theta}\psi}}$ has finite rank if and only if $k_2$ satisfies one of the conditions (2), (4) and (6)-(8) and $$\widehat{\varphi}(z)=
\frac{C}{ z+k_2} \ \ \textrm{and}\ \  \widehat{\psi}(z)=\frac{C}{ z+k_1+k_2}.$$ In other words,
$k_2>-2$, $\varphi(r)=Cr^{k_2}$ and $\psi(r)=Cr^{k_1+k_2}.$

Similarly, if $k_1=-1$, then $T_{e^{ik_1\theta}r^{k_1}} T_{e^{ik_2\theta}\varphi(r)}-T_{{e^{i(k_1+k_2)\theta}\psi}}$ has finite rank if and only if one of the conditions (2), (4), (5) and (7) holds, which implies $k_2>-k_1-2,$ $\varphi(r)=Cr^{k_2}$ and $\psi(r)=Cr^{k_1+k_2}$, as desired.

Condition (a) together with Theorem~\ref{thm eqvi} and the use of
adjoint operators yields that condition (b) holds. This completes the proof.
\end{proof}

\begin{corollary}\label{cor semimono}  Let $k_1,\;k_2\in\mathbb{Z}$ such that $k_1k_2\neq 0$, and let $m_1,\;m_2\in \mathbb{R}$ such that greater than or equal to $-1$.
Then $T_{e^{ik_1\theta }r^{m_1}}T_{e^{ik_2\theta }r^{m_2}}-T_{{e^{i(k_1+k_2)\theta}\psi}}$ has finite rank on $L_h^2$ if and only if
one of the following conditions holds:
\begin{enumerate}
\item[(1)]$k_1=k_2$, $m_1=m_2$, $|k_1|<m_1+2$ and $\psi(r)=\frac{m_1+|k_1|}{2|k_1|}r^{m_1+|k_1|}-\frac{m_1-|k_1|}{2|k_1|}r^{m_1-|k_1|}.$
\item[(2)]$k_1=m_1,\ k_2=m_2,\ k_1+k_2\neq-2$ and $\psi(r)=r^{k_1+k_2}$.
\item[(3)]$k_1=-m_1,\ k_2=-m_2,\ k_1+k_2\neq2$ and $\psi(r)=r^{-k_1-k_2}$.
\end{enumerate}
\end{corollary}
\begin{proof} First suppose $T_{e^{ik_1\theta }r^{m_1}}T_{e^{ik_2\theta }r^{m_2}}-T_{{e^{i(k_1+k_2)\theta}\psi}}$ has finite rank. Then by Corollary~\ref{cor relation} and Corollary~\ref{cor comr},
one of the following cases holds.

\emph{Case} 1. $k_1=k_2$ and $m_1=m_2$. Since $k_1k_2\neq 0$, we get that condition (7) of Theorem~\ref{thm semicom} is the only one which satisfies this case, and hence $|k_1|<m_1+2$. Notice that $\varphi(r)=r^{m_1}$, then we get
\begin{align*}
    \widehat{\psi}(z)&=\frac{z}{\left(z+m_1-|k_1|\right)\left(z+m_1+|k_1|\right)}=\frac{m_1+|k_1|}{2|k_1|}\frac{1}{z+m_1+|k_1|}-\frac{m_1-|k_1|}{2|k_1|}\frac{1}{z+m_1-|k_1|},
\end{align*}
which implies $\psi(r)=\frac{m_1+|k_1|}{2|k_1|}r^{m_1+|k_1|}-\frac{m_1-|k_1|}{2|k_1|}r^{m_1-|k_1|}.$

\emph{Case} 2. $k_1=m_1,\ k_2=m_2$. Notice that $m_1\geq-1$, $m_2\geq-1$ and $k_1k_2\neq 0$, then
\begin{equation}\label{keq}
    k_2>\max\{-2,-k_1-2\}\Longleftrightarrow k_1+k_2\neq-2.
\end{equation}
So by condition (a) of Corollary~\ref{cor semianaly}, we get $k_1+k_2\neq-2$ and $\psi(r)=r^{k_1+k_2}$.

\emph{Case} 3. $k_1=-m_1,\ k_2=-m_2$. Then $T_{e^{ik_1\theta }r^{m_1}}T_{e^{ik_2\theta }r^{m_2}}-T_{{e^{i(k_1+k_2)\theta}\psi}}$ can be written as
$T_{e^{-i(-k_1)\theta }r^{-k_1}}T_{e^{ik_2\theta }r^{-k_2}}-T_{{e^{i(k_1+k_2)\theta}\psi}}.$
Since $k_1\leq 1$, $k_2\leq 1$, we get $$k_2<\min\{2,-k_1+2\}\Longleftrightarrow k_1+k_2\neq 2.$$
So by condition (b) of Corollary~\ref{cor semianaly}, we get $k_1+k_2\neq 2$ and $\psi(r)=r^{-k_1-k_2}$.

By condition (7) of Theorem~\ref{thm semicom} and Corollary~\ref{cor semianaly}, the converse implication is clear.  This completes the proof.
\end{proof}

\begin{corollary}\label{cor semicom} Let $k_1,\;k_2\in\mathbb{Z}$ such that $k_1k_2\neq 0$, and let $m\in \mathbb{R}$, $m\geq -1$.
Then for a nonzero radial T-function $\varphi(r)$ on $D$, the semicommutator $\left(T_{e^{ik_1\theta} r^m},T_{e^{ik_2\theta }\varphi}\right]$ has finite rank on $L_h^2$ if and only if one of the following conditions holds:
\begin{enumerate}
\item[(1)]$m=k_1$, $k_2>\max\{-2,-k_1-2\}$ and $\varphi(r)=r^{k_2}$.
\item[(2)]$m=-k_1$, $k_2<\min\{2,-k_1+2\}$ and $\varphi(r)=r^{-k_2}$.
\end{enumerate}
\end{corollary}
\begin{proof}
If the semicommutator $\left(T_{e^{ik_1\theta} r^m},T_{e^{ik_2\theta }\varphi}\right]$ has finite rank, then it follows from \eqref{pmax1} that $$(z-k_1+k_2)\widehat{r^m}(z+k_2)\widehat{\varphi}(z-k_1)=\widehat{r^m\varphi}(z),$$ which is equivalent to
$(z-k_1+k_2)\widehat{\varphi}(z-k_1)=(z+m+k_2)\widehat{\varphi}(z+m).$
Using the same argument as in the proof of \cite[Theorem 6]{CR}, we can conclude that
\begin{equation}\label{semiva}
\widehat{\varphi}(z-k_1)=\frac{C}{z-k_1+k_2}
\end{equation}
or $m+k_1=0$ holds.

If \eqref{semiva} holds, then $\varphi(r)=Cr^{k_2}$. Notice that $\varphi(r)$ is a T-function, and hence $k_2\geq-1$. Then by Corollary~\ref{cor semimono} we get $m=k_1$ and $k_1+k_2\neq-2.$
Because of \eqref{keq}, we have that condition (1) holds.

If $m+k_1=0$, then $\left(T_{e^{ik_1\theta} r^m},T_{e^{ik_2\theta }\varphi}\right]$
can be written as $T_{e^{-i(-k_1)\theta} r^{-k_1}}T_{e^{ik_2\theta }\varphi}-T_{e^{i(k_1+k_2)\theta }r^{-k_1}\varphi}$
Notice that $-k_1=m\geq-1$, and $k_1k_2\neq 0$,
then Corollary~\ref{cor semianaly} shows that condition (2) holds.

By Corollary~\ref{cor semianaly}, the converse implication is clear. This completes the proof.
\end{proof}

\section{Ranks of commutators and semicommutators on $L_a^2$}

On the Bergman space, $\check{\textrm{C}}\textrm{u}\check{\textrm{c}}\textrm{kovi}\acute{\textrm{c}}$ and Louhichi \cite{CuL} proved that
if two quasihomogeneous Toeplitz operators have both positive or both negative
degrees and if their commutator (or semicommutator) has finite rank, then the commutator (or semicommutator) must be zero.
Moreover, they presented two examples of nonzero finite rank commutators and semicommutators of two quasihomogeneous Toeplitz operators of opposite degrees:
$$rank\left(\left[T_{e^{2i\theta} r^6},T_{e^{-i\theta }(\frac{3}{r}-r^3)}\right]\right)=1\ \ \textrm{and}\ \
rank\left(\left(T_{e^{i\theta}\frac{1}{r}},T_{e^{-i\theta }r}\right]\right)=1.$$
In this section, we continue this line of investigation and study finite rank commutators and generalized semicommutators of
quasihomogeneous Toeplitz operators on the Bergman space, with one of the symbols being of the form $e^{ik\theta} r^m$.
We will show that the corresponding results are different from those of harmonic
Bergman space.

The following lemma is from \cite[Lemma 5.3]{LSZ} which we shall use often in this section.
\begin{lemma}\label{lem Bcalc} Let $k\in\mathbb{Z}$ and let $\varphi$ be a
radial T-function. Then on $L_a^2$, for each $n\in \mathbb{N}$ we have
$$
T_{ e^{ik\theta}\varphi }(z^{n})=\left\{ {\begin{array}{ll}
   {2(n+k+1)\widehat{\varphi }(2n+k+2)z^{n+k}}, & \text{if\; $n\geq -k$},  \\
   {0},
   & \text{if\; $n<-k$.}  \\
\end{array}} \right.\\
$$
\end{lemma}

Unlike harmonic Bergman space case, the following theorem shows that the commutator $\left[T_{e^{ik_1\theta} r^m},T_{e^{ik_2\theta }\varphi}\right]$ on $L_a^2$ can only has finite
rank 0 or 1.
\begin{theorem}\label{thm Bcommute} Let $k_1,\;k_2\in\mathbb{Z}$, and let $m\in \mathbb{R}$, $m\geq -1$.
Then for a nonzero radial T-function $\varphi(r)$ on $D$, the commutator $\left[T_{e^{ik_1\theta} r^m},T_{e^{ik_2\theta }\varphi}\right]$ has finite rank on $L_a^2$ if and only if one of the following conditions holds:
\begin{enumerate}
\item[(1)]$k_1=m=0$.
\item[(2)]$k_2=0$ and $\varphi=C$.
\item[(3)]$k_1=k_2=0$.
\item[(4)]$k_1k_2>0$, $|k_2|<m+|k_1|+2$ and $\widehat{\varphi}(z)=
C\frac{\Gamma\left(\frac{z+|k_2|}{2|k_1|}\right)\Gamma\left(\frac{z+m+|k_1|-|k_2|}{2|k_1|}\right)}{\Gamma\left(\frac{z+2|k_1|-|k_2|}{2|k_1|}\right)
\Gamma\left(\frac{z+m+|k_1|+|k_2|}{2|k_1|}\right)}.$
\item[(5)]$k_1k_2>0$, $|k_2|\geq m+|k_1|+2$, $m=(2n+1)|k_1|$ for some $n\in\mathbb{N}$ and
$$\widehat{\varphi}(z)=C\frac{\left(\frac{z-|k_2|}{2|k_1|}+1\right)\left(\frac{z-|k_2|}{2|k_1|}+2\right)\cdots\left(\frac{z-|k_2|}{2|k_1|}+n\right)}
{\frac{z+|k_2|}{2|k_1|}\left(\frac{z+|k_2|}{2|k_1|}+1\right)\cdots\left(\frac{z+|k_2|}{2|k_1|}+n\right)}.$$
\item[(6)]$k_1k_2<0$, $|k_2|\geq 2$, $m+|k_1|=0$ and ${\varphi}(r)=Cr^{|k_2|}$.
\item[(7)]$k_1k_2<0$, $|k_2|=1$ and $\widehat{\varphi}(z)=
C\frac{\Gamma\left(\frac{z-1}{2|k_1|}\right)\Gamma\left(\frac{z+m+|k_1|+1}{2|k_1|}\right)}{\Gamma\left(\frac{z+2|k_1|+1}{2|k_1|}\right)
\Gamma\left(\frac{z+m+|k_1|-1}{2|k_1|}\right)}.$
\end{enumerate}
In each condition $(1)$-$(5)$, $rank\left(\left[T_{e^{ik_1\theta} r^m},T_{e^{ik_2\theta }\varphi}\right]\right)=0$. In each condition $(6)$-$(7)$,
$$Ran\left(\left[T_{e^{ik_1\theta} r^m},T_{e^{ik_2\theta }\varphi}\right]\right)=\emph{span}\left\{z^{\max\left\{k_1,k_2\right\}-1}\right\}$$
and $$\left[T_{e^{ik_1\theta} r^m},T_{e^{ik_2\theta }\varphi}\right]=C z^{\max\left\{k_1,k_2\right\}-1} \otimes z^{\max\left\{-k_1,-k_2\right\}-1}.$$ Therefore, $rank\left(\left[T_{e^{ik_1\theta} r^m},T_{e^{ik_2\theta }\varphi}\right]\right)=1.$

\end{theorem}
\begin{proof} Proceeding as in the proof of \cite[Theorem 6]{CuL}, we can get that the commutator $\left[T_{e^{ik_1\theta} r^m},T_{e^{ik_2\theta }\varphi}\right]$ on $L_a^2$ has finite rank if and only if
$\left[T_{e^{ik_1\theta} r^m},T_{e^{ik_2\theta }\varphi}\right](z^n)=0$ holds for any natural number $n\geq N_1$,
or equivalently,
\begin{align}\label{brm}
&2(n+k_2+1)\widehat{r^m}(2n+k_1+2k_2+2)\widehat{\varphi}(2n+k_2+2)\nonumber\\
&=2(n+k_1+1)\widehat{r^m}(2n+k_1+2)\widehat{\varphi}(2n+2k_1+k_2+2).
\end{align} Notice that \eqref{brm} is the same as \eqref{rm}, so from the proof of Theorem~\ref{thm commute} we know that \eqref{brm} holds if and only if one of the conditions (1)-(7) holds, as desired.

If one of the conditions (1)-(5) holds, then it is known that $rank\left(\left[T_{e^{ik_1\theta} r^m},T_{e^{ik_2\theta }\varphi}\right]\right)=0.$
Now, we suppose condition (6) holds. Without loss of generality, we can assume that $k_1>0$, then $k_1=1$ and $N_1=-k_2.$
Moreover, for any $0\leq l<-k_2$, by Lemma~\ref{lem Bcalc} we get
\begin{align*}
\left[T_{e^{i\theta} r^{-1}},T_{e^{ik_2\theta}\varphi}\right]({z}^{l})
&=\left\{
{\begin{array}{ll}
   {0}, & \text{if\; $0\leq l<-k_2-1$},  \\
   {-{2(-k_2+1)}\widehat{r^{-1}}(-2k_2+1)2\widehat{\varphi }(-k_2+2)},
   & \text{if\; $l=-k_2-1$,}
\end{array}} \right.\\
&=\left\{
{\begin{array}{ll}
   {0}, & \text{if\; $0\leq l<-k_2-1$},  \\
   {\frac{1}{k_2}},
   & \text{if\; $l=-k_2-1$.}
\end{array}} \right.
\end{align*}
Thus,
$Ran\left(\left[T_{e^{i\theta} r^{-1}},T_{e^{ik_2\theta}\varphi}\right]\right)=\emph{span}\left\{1\right\} $ and$\left[T_{e^{ik_1\theta} r^m},T_{e^{ik_2\theta }\varphi}\right]=-\left(1 \otimes z^{-k_2-1}\right).$
Next, we suppose condition (7) holds.  We can also assume that $k_1>0$, then $k_2=-1$ and $N_1=1.$
Moreover, by Lemma~\ref{lem Bcalc} we get
$$
\left[T_{e^{ik_1\theta} r^{m}},T_{e^{-i\theta }\varphi}\right]({z}^{0})=-4k_1(k_1+1)\widehat{r^m}(k_1+2)\widehat{\varphi }(2k_1+1)z^{k_1-1}.
$$
Thus,
$Ran\left(\left[T_{e^{i\theta} r^{-1}},T_{e^{ik_2\theta}\varphi}\right]\right)=\emph{span}\left\{z^{k_1-1}\right\}$
and
$$\left[T_{e^{ik_1\theta} r^{m}},T_{e^{-i\theta }\varphi}\right]=-4k_1(k_1+1)\widehat{r^m}(k_1+2)\widehat{\varphi }(2k_1+1) z^{k_1-1}\otimes 1.$$
This completes the proof.
\end{proof}

We now observe several consequences of Theorem~\ref{thm Bcommute}. First, we give two examples of the commutators on $L_a^2$ with nonzero finite rank, corresponding to the conditions (6) and (7) of Theorem~\ref{thm Bcommute}.

\begin{example} Let $\varphi(r)$ be a nonzero radial T-function on $D$, then on $L_a^2$, the following statements hold.
\begin{enumerate}
\item[(1)] The commutator $\left[T_{e^{2i\theta} r^{6}},T_{e^{-i\theta}\varphi(r)}\right]$ has finite rank if and only if $\varphi(r)=C\left(\frac{3}{r}-r^3\right)$.\\ In this case, $\left[T_{e^{2i\theta} r^{6}},T_{e^{-i\theta}\varphi(r)}\right]=-\frac{3}{2}C\left(z\otimes1\right)$.
\item[(2)] The commutator $\left[T_{e^{i\theta} r^{-1}},T_{e^{-3i\theta}\varphi(r)}\right]$ has finite rank if and only if $\varphi(r)=Cr^3$.\\ In this case,
$\left[T_{e^{i\theta} r^{-1}},T_{e^{-3i\theta}\varphi(r)}\right]=-C\left(1 \otimes z^{2}\right)$.
\end{enumerate}
\end{example}

\begin{corollary}  Let $k_1,\;k_2\in\mathbb{Z}$ such that $k_1>0$ or $k_1=-1$,
and let $\varphi(r)$ be a nonzero radial T-function on $D$. Then on $L_a^2$, the following statements hold.
\begin{enumerate}
\item[(a)] The commutator $\left[T_{e^{ik_1\theta }r^{k_1}},T_{e^{ik_2\theta }\varphi}\right]$ has finite rank if and only if
$$ k_2>-2\ \ \textrm{and}\ \  \varphi(r)=Cr^{k_2}.$$
\item[(b)] The commutator $\left[T_{e^{-ik_1\theta }r^{k_1}},T_{e^{ik_2\theta }\varphi}\right]$ has finite rank if and only if
$$ k_2<2\ \ \textrm{and}\ \  \varphi(r)=Cr^{-k_2}.$$
\end{enumerate}
\end{corollary}
\begin{proof} If $k_1>0$, then by Theorem~\ref{thm Bcommute}, the commutator $\left[T_{e^{ik_1\theta }r^{k_1}},T_{e^{ik_2\theta }\varphi}\right]$ has finite rank if and only if $k_2$ satisfies one of the conditions (2), (4), (5) and (7) and $$\widehat{\varphi}(z)=
\frac{C}{ z+k_2}.$$ In other words, $k_2>-2$ and $\varphi(r)=Cr^{k_2}$.
Similarly, if $k_1=-1$, then one of the conditions (2) and (4), (6) and (7) holds, and hence $k_2>-2$ and $\varphi(r)=Cr^{k_2}$.

Combining condition (a) with the use of
adjoint operators, one can get condition (b) holds.
\end{proof}

\begin{corollary}  Let $k_1,\;k_2\in\mathbb{Z}$ such that $k_1k_2\neq 0$, and let $m_1,\;m_2\in \mathbb{R}$ such that greater than or equal to $-1$.
Then the commutator $\left[T_{e^{ik_1\theta }r^{m_1}},T_{e^{ik_2\theta }r^{m_2}}\right]$ has finite rank on $L_a^2$ if and only if
one of the following conditions holds:
\begin{enumerate}
\item[(1)]$k_1=k_2$ and $m_1=m_2$.
\item[(2)]$k_1=m_1$ and $k_2=m_2$.
\item[(3)]$k_1=-m_1$ and $k_2=-m_2$.
\end{enumerate}
\end{corollary}
\begin{proof} The proof is similar to Corollary~\ref{cor comr}.
\end{proof}

Next, we will discuss the finite rank generalized semicommutators of quasihomogeneous Toeplitz
operators on $L_a^2$. First, we give the following result.

\begin{proposition}\label{prop Bpva} Let $k_1,\;k_2\in\mathbb{Z}$, and let $m\in \mathbb{R}$, $m\geq -1$.
Then for radial T-functions $\varphi$ and $\psi$ on $D$, the generalized semicommutator $T_{e^{ik_1\theta} r^m}T_{e^{ik_2\theta }\varphi}-T_{e^{i(k_1+k_2)\theta }\psi}$ has finite rank on $L_a^2$ if and only if
\begin{equation}\label{psirm}
\psi(r)=\frac{\varphi(r)}{r^{k_1}}-(m+k_1)r^{m+k_2}\int^1_{r}{\frac{\varphi(t)}{t^{m+k_1+k_2+1}}}{dt}
\end{equation}
\end{proposition}
\begin{proof} Using the same argument as in the proof of Proposition~\ref{prop commutemax}, we see that $T_{e^{ik_1\theta} r^m}T_{e^{ik_2\theta }\varphi}-T_{e^{i(k_1+k_2)\theta }\psi}$ on $L_a^2$ has finite rank if and only if
$$
2(n+k_2+1)\widehat{r^m}(2n+k_1+2k_2+2)\widehat{\varphi}(2n+k_2+2)=\widehat{\psi }(2n+k_1+k_2+2)
$$
holds for any natural number $n\geq N_2$, and hence
$${\left(r^{m+k_1+2k_2+2N_2}\right) }\ast_{M}{\left(r^{k_2+2N_2}\varphi\right) }
={\left(r^{2k_2+2N_2}\right)}\ast_{M}{\left(r^{k_1+k_2+2N_2}\psi\right) },
$$
which is equivalent to $$r^{m+k_1}\int^1_{r}{\frac{\varphi(t)}{t^{m+k_1+k_2+1}}}{dt}=\int^1_{r}{\frac{\psi(t)}{t^{k_2-k_1+1}}}{dt}.$$
By differentiating both sides, we get \eqref{psirm} holds.
\end{proof}

\begin{remark} Let $k_1,\;k_2,\;m$ and $\varphi$ be as in Proposition~\ref{prop Bpva}, and let $\psi$ be defined by \eqref{psirm}. If $\frac{\varphi(r)}{r^{k_1}}$is a T-function and either $m+k_2\geq-1$ or $m+k_1=0$ hold, then $\psi$ is a T-function, and hence $T_{e^{ik_1\theta} r^m}T_{e^{ik_2\theta }\varphi}-T_{e^{i(k_1+k_2)\theta }\psi}$ has finite rank.
In fact, since $\frac{\varphi(r)}{r^{k_1}}$ is a T-function, the desired result is obvious when $m+k_1=0$. Now we suppose $m+k_2\geq-1$.
Then we have
\begin{align*}
\int_0^1\left|r^{m+k_2}\int^1_{r}{\frac{\varphi(t)}{t^{m+k_1+k_2+1}}}{dt}\right|rdr&\leq \int_0^1{r^{m+k_2+1}dr\int^1_{r}{\frac{\left|\varphi(t)\right|}{t^{m+k_1+k_2+1}}}{dt}}\\
&=\int_0^1\frac{\left|\varphi(t)\right|}{t^{m+k_1+k_2+1}}{dt}\int_0^tr^{m+k_2+1}dr\\
&=\frac{1}{m+k_2+2}\left\|\frac{\varphi(r)}{r^{k_1}}\right\|_{L^{1}}<\infty.
\end{align*}
Moreover,
$$\left|r^{m+k_2}\int^1_{r}{\frac{\varphi(t)}{t^{m+k_1+k_2+1}}}{dt}\right|
\leq\int_r^1{{\left|\frac{\varphi(t)}{t^{k_1}}\right|}}\frac{dt}{t}\leq\frac{1}{r^2}\left\|\frac{\varphi(r)}{r^{k_1}}\right\|_{L^{1}}.
$$
Therefore, the radial function $r^{m+k_2}\int^1_{r}{\frac{\varphi(t)}{t^{m+k_1+k_2+1}}}{dt}$ is "nearly bounded" on $D$, and hence $\psi$ is a T-function.
\end{remark}

The following theorem completely characterizes the finite rank generalized semicommutators of two Toeplitz
operators on $L^2_a$ with monomial symbols.

\begin{theorem} Let $k_1,\;k_2\in\mathbb{Z}$, and let $m_1,\;m_2\geq -1$.
Then for a radial T-function $\psi$ on $D$, the generalized semicommutator $T_{e^{ik_1\theta} r^{m_1}}T_{e^{ik_2\theta }r^{m_2}}-T_{e^{i(k_1+k_2)\theta }\psi}$ has finite rank on $L^2_a$ if and only if
\begin{equation}\label{Bprl}
\psi(r)=\left\{
{\begin{array}{ll}
   {\frac{k_1+m_1}{m_1+k_2-m_2+k_1}r^{m_1+k_2}-\frac{m_2-k_2}{m_1+k_2-m_2+k_1}r^{m_2-k_1}}, & \text{if\; $m_1+k_2\neq m_2-k_1$},  \\
   {r^{m_1+k_2}\left[1+\left(m_1+k_1\right)\log r\right]},
   & \text{if\; $m_1+k_2=m_2-k_1$,}
\end{array}} \right.
\end{equation}
and one of the following conditions holds:
\begin{enumerate}
\item[(1)]$k_1+m_1\neq 0$, $k_2-m_2\neq 0$, $m_1+k_2\geq -1$ and $m_2-k_1\geq -1$.
\item[(2)]$k_1+m_1=0$ and $m_2-k_1\geq -1$.
\item[(3)]$k_2-m_2=0$ and $m_1+k_2\geq -1$.
\end{enumerate}
Furthermore, $T_{e^{ik_1\theta} r^{m_1}}T_{e^{ik_2\theta }r^{m_2}}-T_{e^{i(k_1+k_2)\theta }\psi}$ on $L^2_a$ has nonzero finite rank if and only if \eqref{Bprl} holds and one of the following conditions holds:
\begin{itemize}
  \item $k_1+k_2\geq0$, $k_2\leq-2$ and condition $(1)$ holds. In this case,
$$Ran\left(T_{e^{ik_1\theta} r^{m_1}}T_{e^{ik_2\theta }r^{m_2}}-T_{e^{i(k_1+k_2)\theta }\psi}\right)=\emph{span}\left\{z^{k_1+k_2},\cdots, z^{k_1-2}\right\}$$
and $$T_{e^{ik_1\theta} r^{m_1}}T_{e^{ik_2\theta }r^{m_2}}-T_{e^{i(k_1+k_2)\theta }\psi}=\sum\limits_{j = 0}^{-k_2-2} {C_{j}\left( z^{j+k_1+k_2} \otimes z^j \right)}$$
for nonzero constant $C_{j}$. Therefore,
$rank\left(T_{e^{ik_1\theta} r^{m_1}}T_{e^{ik_2\theta }r^{m_2}}-T_{e^{i(k_1+k_2)\theta }\psi}\right)=-k_2-1.$
  \item $k_1+k_2\geq0$, $k_2=-1$ and one of the conditions $(2)$-$(3)$ holds. In this case,
$$Ran\left(T_{e^{ik_1\theta} r^{m_1}}T_{e^{ik_2\theta }r^{m_2}}-T_{e^{i(k_1+k_2)\theta }\psi}\right)=\emph{span}\left\{z^{k_1-1}\right\}$$
and $$T_{e^{ik_1\theta} r^{m_1}}T_{e^{ik_2\theta }r^{m_2}}-T_{e^{i(k_1+k_2)\theta }\psi}={C\left( z^{k_1-1} \otimes 1 \right)}$$
for nonzero constant $C$. Therefore,
$rank\left(T_{e^{ik_1\theta} r^{m_1}}T_{e^{ik_2\theta }r^{m_2}}-T_{e^{i(k_1+k_2)\theta }\psi}\right)=1.$
  \item $k_1+k_2<0$, $k_1\geq 2$ and condition $(1)$ holds. In this case,
$$Ran\left(T_{e^{ik_1\theta} r^{m_1}}T_{e^{ik_2\theta }r^{m_2}}-T_{e^{i(k_1+k_2)\theta }\psi}\right)=\emph{span}\left\{1,\cdots, z^{k_1-2}\right\}$$
and $$T_{e^{ik_1\theta} r^{m_1}}T_{e^{ik_2\theta }r^{m_2}}-T_{e^{i(k_1+k_2)\theta }\psi}=\sum\limits_{j = 0}^{k_1-2} {C_{j}\left( z^{j} \otimes z^{j-k_1-k_2} \right)}$$
for nonzero constant $C_{j}$. Therefore,
$rank\left(T_{e^{ik_1\theta} r^{m_1}}T_{e^{ik_2\theta }r^{m_2}}-T_{e^{i(k_1+k_2)\theta }\psi}\right)=k_1-1.$
  \item $k_1+k_2<0$, $k_1=1$ and condition $(2)$ holds. In this case,
$$Ran\left(T_{e^{ik_1\theta} r^{m_1}}T_{e^{ik_2\theta }r^{m_2}}-T_{e^{i(k_1+k_2)\theta }\psi}\right)=\emph{span}\left\{1\right\}$$
and $$T_{e^{ik_1\theta} r^{m_1}}T_{e^{ik_2\theta }r^{m_2}}-T_{e^{i(k_1+k_2)\theta }\psi}= {C\left( 1 \otimes z^{-k_1-k_2} \right)}$$
for nonzero constant $C$. Therefore,
$rank\left(T_{e^{ik_1\theta} r^{m_1}}T_{e^{ik_2\theta }r^{m_2}}-T_{e^{i(k_1+k_2)\theta }\psi}\right)=k_1.$
\end{itemize}
\end{theorem}
\begin{proof} By Proposition~\ref{prop Bpva}, $T_{e^{ik_1\theta} r^{m_1}}T_{e^{ik_2\theta }r^{m_2}}-T_{e^{i(k_1+k_2)\theta }\psi}$ has finite rank if and only if
\begin{align*}
\psi(r)&=\frac{r^{m_2}}{r^{k_1}}-(m_1+k_1)r^{m_1+k_2}\int^1_{r}{\frac{t^{m_2}}{t^{m_1+k_1+k_2+1}}}{dt}\\
&=\left\{
{\begin{array}{ll}
   {\frac{k_1+m_1}{m_1+k_2-m_2+k_1}r^{m_1+k_2}-\frac{m_2-k_2}{m_1+k_2-m_2+k_1}r^{m_2-k_1}}, & \text{if\; $m_1+k_2\neq m_2-k_1$},  \\
   {r^{m_1+k_2}\left[1+\left(m_1+k_1\right)\log r\right]},
   & \text{if\; $m_1+k_2=m_2-k_1$},
\end{array}} \right.
\end{align*}
and $\psi(r)$ is a T-function. It is also clear that $\psi(r)$ defined by \eqref{Bprl} is a T-function if and only if one of the conditions (1)-(3) holds.

Next, we need to determine the range of the generalized semicommutator $T_{e^{ik_1\theta} r^{m_1}}T_{e^{ik_2\theta }r^{m_2}}-T_{e^{i(k_1+k_2)\theta }\psi}$. We have shown in Proposition~\ref{prop Bpva} that $\left(T_{e^{ik_1\theta} r^{m_1}}T_{e^{ik_2\theta }r^{m_2}}-T_{e^{i(k_1+k_2)\theta }\psi}\right)(z^n)=0$ holds for any natural number $n\geq N_2$.

Now, we assume that $k_1+k_2\geq0$. If $k_2\geq0$, then $N_2=0$, and hence $$rank\left(T_{e^{ik_1\theta} r^{m_1}}T_{e^{ik_2\theta }r^{m_2}}-T_{e^{i(k_1+k_2)\theta }\psi}\right)=0.$$
If $k_2\leq-1$, then $N_2=-k_2$. For any $0\leq n<-k_2$, by Lemma~\ref{lem Bcalc} we get
\begin{align*}
    &\left[T_{e^{ik_1\theta} r^{m_1}}T_{e^{ik_2\theta }r^{m_2}}-T_{e^{i(k_1+k_2)\theta }\psi}\right]({z}^{n})=2( n+k_1+k_2+1)\widehat{\psi}(2n+k_1+k_2+2)z^{n+k_1+k_2}.
\end{align*}
On the other hand, \eqref{Bprl} implies that
$$\widehat{\psi}(z)=\left\{
{\begin{array}{ll}
   {\frac{1}{m_1+k_2-m_2+k_1}\left(\frac{k_1+m_1}{m_1+k_2+z}-\frac{m_2-k_2}{m_2-k_1+z}\right)}, & \text{if\; $m_1+k_2\neq m_2-k_1$},  \\
   {\frac{z-k_1+k_2}{(m_1+k_2+z)^2}},
   & \text{if\; $m_1+k_2=m_2-k_1$,}
\end{array}} \right.$$
which is well defined on $\left\{z:
\textrm{Re}z\geq 2\right\}$.
Then a direct calculation shows that
\begin{align*}
    &\widehat{\psi}(2n+k_1+k_2+2)=0\Longleftrightarrow n=-k_2-1, k_1+m_1\neq 0\  \textrm{and}\ k_2-m_2\neq 0.
\end{align*}
Thus, if condition (1) holds, then $$rank\left(T_{e^{ik_1\theta} r^{m_1}}T_{e^{ik_2\theta }r^{m_2}}-T_{e^{i(k_1+k_2)\theta }\psi}\right)=\max\{0, -k_2-1\},$$
and if either condition (2) or condition (3) holds, then $$rank\left(T_{e^{ik_1\theta} r^{m_1}}T_{e^{ik_2\theta }r^{m_2}}-T_{e^{i(k_1+k_2)\theta }\psi}\right)=-k_2.$$
Since $-k_1\leq k_2\leq-1$ and $m_1,\;m_2\geq -1$, then both condition (2) and condition (3) imply that $k_2=-1$, as desired.

If $k_1+k_2<0$, then by using the adjoint operator, the desired results are obvious. This completes the proof.
\end{proof}

Below we present the examples of nonzero generalized semicommutators of finite rank on $L_a^2$.
\begin{example}\label{ex commute}  Let $\psi$ be a radial T-function on $D$, then on $L_a^2$, the following statements hold.
\begin{enumerate}
\item[(1)] $T_{e^{2i\theta} r}T_{e^{-2i\theta}r^2}-T_{{\psi}}$ has finite rank if and only if
$\psi(r)=4-3r^{-1}$.\\ In this case, $T_{e^{2i\theta} r}T_{e^{-2i\theta}r^2}-T_{\psi}=2\left(1\otimes1\right)$.
\item[(2)] $T_{e^{3i\theta}r^5}T_{e^{-i\theta} r^{-1}}-T_{{e^{2i\theta}\psi}}$ has finite rank if and only if
$\psi(r)=r^4$.\\ In this case, $T_{e^{3i\theta}r^5}T_{e^{-i\theta} r^{-1}}-T_{{e^{2i\theta}\psi}}=-\frac{3}{4}\left(z^2\otimes1\right)$.
\item[(3)] $T_{e^{3i\theta}r^5}T_{e^{-4i\theta} r^{3}}-T_{{e^{-i\theta}\psi}}$ has finite rank if and only if
$\psi(r)=8r-7$.\\ In this case, $T_{e^{4i\theta} r^{3}}T_{e^{-3i\theta}r^5}-T_{{e^{i\theta}\psi}}=\frac{4}{3}\left(1\otimes z\right)+\frac{4}{5}\left(z\otimes z^2\right)$.
\item[(4)] $T_{e^{i\theta} r^{-1}}T_{e^{-3i\theta}r^3}-T_{{e^{-2i\theta}\psi}}$ has finite rank if and only if
$\psi(r)=r^2$.\\ In this case, $T_{e^{i\theta} r^{-1}}T_{e^{-3i\theta}r^3}-T_{{e^{-2i\theta}\psi}}=-\left(1\otimes z^2\right)$.
\end{enumerate}
\end{example}

\section{Further results and remarks}

There is a general feeling
that the restrictions on their symbols for Toeplitz operators to satisfy certain algebraic properties on $L_h^2$ are more than $L_a^2$, as the orthonormal basis for $L_h^2$ involves co-analytic monomial. We list here some results which illustrate this view.
\begin{enumerate}
\item[(1).] Choe and Lee \cite{CL} showed that two analytic
Toeplitz operators on $L_h^2$ commute only when their symbols and
the constant function $1$ are linearly dependent, but analytic Toeplitz operators always commute on $L_a^2$.
\item[(2).] Guo and Zheng \cite{GZ} obtain a criterion for compactness of Toeplitz operators with bounded symbols on $L_h^2$, which requires one more condition than that on $L_a^2$.
\item[(3).] We showed in \cite{DLZ5} that Toeplitz operators with bounded symbols on $L_h^2$ commute with $T_{e^{ik\theta}r^m}$ only
in trivial cases, but \cite{CR} showed there are many nontrivial Toeplitz operators commuting with $T_{e^{ik\theta}r^m}$ on $L_a^2$.
\item[(4).] We proved in \cite{DZ4} that the product of two Toeplitz operators with bounded symbols $T_{e^{ik_1\theta}\varphi}T_{e^{ik_2\theta}r^m}$ on $L_h^2$ is a Toeplitz operator only in trivial cases, but it also holds on $L_a^2$ when either $\overline{e^{ik_1\theta}\varphi}$ or $e^{ik_2\theta}r^m$ is analytic.
\end{enumerate}
However, comparing to the Bergman space case, the theory of Toeplitz operators on $L_h^2$ has its own unique properties. The most unexpected results to us are the symmetry properties obtained in \cite{CKL,DZ3,DZ4}.
In the following, we will show the relation between finite rank commutators (or generalized semicommutators) of quasihomogenous Toeplitz operators on $L_h^2$ and that on $L_a^2$.

\begin{proposition}\label{prop relationcom}  Let $k_1,\;k_2\in\mathbb{Z}$, and let ${\varphi_{1}},\;{\varphi_{2}}$ and $\psi$ be radial T-functions. Then
the following statements hold.
\begin{enumerate}
  \item The commutator $\left[T_{e^{ik_1\theta} \varphi_{1}},T_{e^{ik_2\theta }\varphi_2}\right]$ has finite rank on $L_h^2$ if and only if $\left[T_{e^{ik_1\theta} \varphi_{1}},T_{e^{ik_2\theta }\varphi_2}\right]$ has finite rank on $L_a^2$.
  \item If the generalized semicommutator $T_{e^{ik_1\theta}\varphi_{1}}T_{e^{ik_2\theta }\varphi_2}-T_{e^{i(k_1+k_2)\theta }\psi}$ has finite rank on $L_h^2$, then both  $T_{e^{ik_1\theta}\varphi_{1}}T_{e^{ik_2\theta }\varphi_2}-T_{e^{i(k_1+k_2)\theta }\psi}$ and $T_{e^{ik_2\theta }\varphi_2}T_{e^{ik_1\theta}\varphi_{1}}-T_{e^{i(k_1+k_2)\theta }\psi}$ have finite rank on $L_a^2$.
\end{enumerate}
\end{proposition}
\begin{proof} The same reasoning as in the proof of \cite[Theorem 6]{CuL} shows that $\left[T_{e^{ik_1\theta} \varphi_{1}},T_{e^{ik_2\theta }\varphi_2}\right]$ has finite rank on $L_a^2$ if and only if \eqref{cmax} holds for any natural number $n\geq N_1$. Then it follows from Proposition~\ref{prop commutemax} that condition $(1)$ holds.

Now, we suppose $T_{e^{ik_1\theta}\varphi_{1}}T_{e^{ik_2\theta }\varphi_2}-T_{e^{i(k_1+k_2)\theta }\psi}$ has finite rank on $L_h^2$. Then by Proposition~\ref{prop pruductmax}, we get \eqref{pmax1} holds for any natural number $n\geq N_2$ and \eqref{pmax2} holds for any natural number $n\geq N_3$. On the other hand,
the same reasoning as in the proof of \cite[Theorem 4]{CuL} shows that $T_{e^{ik_1\theta}\varphi_{1}}T_{e^{ik_2\theta }\varphi_2}-T_{e^{i(k_1+k_2)\theta }\psi}$ has finite rank on $L_a^2$ if and only if \eqref{pmax1} holds for $n\geq N_2$, and $T_{e^{ik_2\theta }\varphi_2}T_{e^{ik_1\theta}\varphi_{1}}-T_{e^{i(k_1+k_2)\theta }\psi}$ has finite rank on $L_a^2$ if and only if
$$
2(n+k_1+1)\widehat{\varphi_1}(2n+k_1+2)\widehat{\varphi_2}(2n+2k_1+k_2+2)=\widehat{\psi }(2n+k_1+k_2+2)
$$
holds for any natural number $n\geq \max\left\{0,-k_1,-k_1-k_2\right\}$.
Replacing $n$ by $n-k_1-k_2$, we see that the above equation is the same as
\eqref{pmax2}. This completes the proof.
\end{proof}

\begin{remark}
(1) Suppose $k_1,\;k_2,\;m,\;\varphi$ and $\psi$ satisfy one of the conditions in Theorem~\ref{thm semicom}. Then as a direct consequence of Proposition~\ref{prop relationcom}, we get that $T_{e^{ik_1\theta} r^m}T_{e^{ik_2\theta }\varphi}-T_{e^{i(k_1+k_2)\theta }\psi}$ has finite rank on $L_a^2$.

(2) In general, the rank of finite rank commutator or generalized semicommutator of quasihomogenous Toeplitz operators on $L_h^2$ should be greater than or equal to that on $L_a^2$. However, it is not all the cases. For example, we have shown in Theorem~\ref{thm commute} and Theorem~\ref{thm semicom} that if $m\in \mathbb{R}$, $m\geq -1$ and $\varphi(r)=\frac{m+1}{2}r^{-1}-\frac{m-1}{2}r,$ then on $L_h^2$,
$$T_{e^{i\theta} r^m}T_{e^{-i\theta }\varphi}=T_{e^{-i\theta }\varphi}T_{e^{i\theta} r^m}$$ and $$T_{e^{i\theta} r^m}T_{e^{-i\theta }\varphi}=T_{1}.$$
However, by a direct calculation, one can get that on $L_a^2$,
$$\left[T_{e^{i\theta} r^m},T_{e^{-i\theta }\varphi}\right]=-\left(T_{e^{i\theta} r^m}T_{e^{-i\theta }\varphi}-T_{1}\right)=1\otimes 1.$$
Therefore, both $\left[T_{e^{i\theta} r^m},T_{e^{-i\theta }\varphi}\right]$ and $T_{e^{i\theta} r^m}T_{e^{-i\theta }\varphi}-T_{1}$ have rank zero on $L_h^2$, and rank one on $L_a^2$.
\end{remark}

\textrm{Acknowledgements:}
This work was supported in part by the National Natural Science Foundation of
China (Grant Nos. 11201331; 11371276; 11401431).

\end{document}